\documentclass[12pt, fleqn]{article}

\pdfoutput=1

\usepackage{amsfonts,amsmath,amsthm,amssymb, xcolor, tikz,calc}
\usepackage{changebar}

\usepackage[scale=0.8]{geometry}

\newtheorem{thm}{Theorem}
\newtheorem{lem}{Lemma}
\newtheorem{cor}{Corollary}

\renewcommand{\Im}{{\rm Im}\,}
\renewcommand{\Re}{{\rm Re}\,}
\newcommand{\cis}{com\-plete inter\-pola\-ting se\-quence}
\renewcommand{\i}{{\rm i}}
\renewcommand{\mod}{{\rm mod}\,}

\begin{document}

\begin{center}
\huge{\bf Complete interpolating sequences, the discrete Muckenhoupt condition, and conformal mapping }
\end{center}
\begin{center}
\rm {\sc  Gunter Semmler}\\[1ex]
Centre of Mathematics,
 Technical University of Munich,\\
Boltzmannstr. 3,
 85747 Garching, Germany 
\end{center}
\vspace{2ex}

\noindent
\begin{small}
We extend the parameterization of sine-type functions in terms of conformal mappings onto slit domains given by Eremenko and Sodin to the more general case of  generating functions  of real  complete interpolating sequences. It turns out that the cuts have to fulfill the discrete Muckenhoupt condition studied earlier by Lyubarskii and Seip. 

\end{small}
\vspace{1ex}

\noindent
Subject Classification (MSC 2000): 42A65,  30D15,  30C20
\vspace{1ex}


\section{Interpolation in Paley-Wiener spaces and Riesz bases of exponentials}
\label{sec1}
Let $PW_\pi^2$ be the \emph{Paley-Wiener space} of all entire functions of exponential type at most $\pi$  which belong to $L^2$ on the real line.  A sequence $\{\lambda\}_{n\in\mathbb Z}$ of complex numbers is called \emph{\cis}  if for all complex  sequences $\{a_n\}_{n\in\mathbb Z}$ with
\begin{equation}
\label{eq1}
\sum_{n\in\mathbb Z} |a_n|^2{\rm e}^{-2\pi |\Im \lambda_n|}(1+|\Im \lambda_n|)<\infty
\end{equation}
the interpolation problem 
\begin{equation}
\label{interpol}
 f(\lambda_n)=a_n
 \end{equation}
has  a unique solution $f\in PW_\pi^2$.   It is known (see \cite{lyuseip}) that for {\cis}s the function $f$ depends continuously on the  sequence $\{a_n\}$ in  the sense that there are constants $C,c>0$ independent of $\{a_n \}$   such that 
\[c\|f\|_{L^2(\mathbb R)}^2\leq \sum_{n\in\mathbb Z} |a_n|^2{\rm e}^{-2\pi|\Im \lambda_n|}(1+|\Im \lambda_n|)\leq C\|f\|_{L^2(\mathbb R)}^2.\]
If $\{\lambda_n\}\subset \mathbb R$ condition (\ref{eq1}) just means that $\{ a_n\}\in l^2(\mathbb Z)$, and in the sequel we will only consider this case.

A sequence $\{e_n\}$ of vectors in a Hilbert space $H$ is called a  \emph{Riesz basis in $H$}  if it is complete and there are $C,c >0$ such that
\[c\sum |a_n|^2\leq \left \|\sum a_n e_n\right \|^2\leq C\sum |a_n|^2\]
 holds for every finite sequence $\{a_n\}$ of numbers. See \cite{chr} for basic information on this topic and equivalent definitions.  Using the fact that the Fourier transform provides an isometry between $L^2(-\pi,\pi)$ and $PW^2_\pi$, one can characterize {\cis}s as those sequences for which the set $\{{\rm e}^{\lambda_n{\rm i} t}\}_{n\in\mathbb Z}$ is a Riesz basis for $L^2(-\pi,\pi)$, cf. Theorem 9 in chapter 4 of \cite{young}.

 It is also well-known that a sequence $\{e_n\}$ of vectors in a Hilbert space $H$  is a Riesz basis in $H$ if and only if it is an unconditional basis with $0<\inf_n \|e_n\|\leq \sup_n \|e_n\|<\infty $, see \cite{gohkrein},\cite{par}. Therefore, some of our cited references deal with unconditional bases of exponentials. 
 
  The problem to describe all {\cis}s  has been studied since the classical book \cite{paleywiener} of Paley and Wiener. The first type of results in this direction  were perturbation statements of the set of integers, and we provide as an example the famous theorem of Kadets \cite{kad}, see also \cite{young}.
 
 \begin{thm} A sequence $\{\lambda_n\}_{n\in\mathbb Z}$ of real numbers is a {\cis} if 
 \[\sup_{n\in\mathbb Z} |\lambda_n-n|<\frac{1}{4}.\]
 The constant $\frac{1}{4}$ cannot be replaced by any larger number. 
 \end{thm}

The full description of all {\cis}s was  obtained by Pavlov \cite{pav}, who proved the following theorem. 
 \renewcommand{\labelenumi}{{\rm (\roman{enumi})}}
  \begin{thm}
 \label{pavlov}
 A sequence $\{\lambda_n\}_{n\in\mathbb Z} $ of real numbers is a \cis{}  if and only if
\begin{enumerate}
\item it is \emph{ separated}, i.e. 
\begin{equation}
\label{sep}
\delta:=\inf_{n\neq m}|\lambda_n-\lambda_m|>0,
\end{equation}
\item the limit\footnote{If one of the numbers $\lambda_n$ is equal to zero, the corresponding factor has to be replaced by $z$. This convention will also be adopted in what follows. } 
\begin{equation}
\label{gf}
F(z)=\lim_{R\to\infty} \prod_{|\lambda_n|< R}\left(1-\frac{z}{\lambda_n}\right)
\end{equation}
exists uniformly on compact subsets of $\mathbb C$ and defines an entire function $F$ of exponential type $\pi$, the \emph{generating function},
\item  the function $w(x):=|F(x+iy)|^2, \; x\in \mathbb R$ satisfies the Muckenhoupt condition
\renewcommand{\theequation}{$A_2$}
\begin{equation}
\label{A2}
\int_I w(x)dx \int_I\frac{1}{w(x)}dx\leq C|I|^2
\end{equation}
{\it for some constant $C>0$, some $y\neq 0$ (and hence for all $y\neq 0$), and all intervals $I\subset \mathbb R$ of finite length $|I|$.  }
\end{enumerate}
\renewcommand{\theequation}{\arabic{equation}}
\addtocounter{equation}{-1}
 \end{thm}
This formulation of the theorem is specialized for real $\lambda_n$, whereas \cite{pav} contains conditions even if the $\lambda_n$ are in some half plane $\{ \Im z\geq h\}$.
Later  this restriction was completely disposed of by Nikolskii \cite{nik} and Minkin \cite{min}. 

Because of the separation condition (\ref{sep}) we can in the following always assume that the $\{\lambda_n\}$ are ordered, i.e. $\lambda_{n+1}>\lambda_n$ for all $n\in\mathbb Z$. Besides, it is also well-known (cf. \cite{lyuseip}) that a (real) {\cis} is \emph{relatively dense}, i.e. for some $\varepsilon>0$ the sets $\{\lambda_n\}_{n\in\mathbb Z}\cap [x-\varepsilon, x+\varepsilon]$ are nonempty for every $x\in\mathbb R$. In particular, 
\begin{equation}
\label{eq2}
\Delta:=\sup_{n\in\mathbb Z} |\lambda_{n+1}-\lambda_n|<\infty.
\end{equation}

Theorem \ref{pavlov} has been deduced in a different manner and generalized by Sedletskii \cite{sed1},  \cite{sed2} and Lyubarskii and Seip \cite{lyuseip}. In the latter paper it is shown that condition (iii) in this theorem can be replaced by 
\begin{enumerate}
\item[(iii')] {\it There is a relatively dense subsequence $ \{\lambda_{n_k}\}_{k\in\mathbb Z}$ such that the numbers $d_k:= |F'(\lambda_{n_k})|^2$ satisfy the \emph{discrete Muckenhoupt condition}}
\renewcommand{\theequation}{$\tilde{A}_2$}
\begin{equation}
\label{A2t}
\sum_{n\in I} d_n \sum_{n\in I} d_n^{-1}\leq C|I|^2
\end{equation}
{\it for a constant $C>0$ and all finite sets $I$ of consecutive integers containing $|I|$ elements. }
\end{enumerate}
\renewcommand{\theequation}{\arabic{equation}}
\addtocounter{equation}{-1}

Checking the Muckenhoupt condition (\ref{A2}) for a function $F$ given by an infinite product (converging in the Cauchy principal value sense) is practically quite hard. Condition (\ref{A2t}) is already easier to verify since it involves only countably many sets $I$ instead of all finite intervals.

The goal of this paper is to give an alternative characterization of {\cis}s. In contrast to existing characterizations our result can be regarded as a parameterization of the set of {\cis}s by independent parameters. Once this is done we can easily give discrete analogues of well-known sufficient conditions for such sequences.    As a by-product we obtain a different way to represent the generating function. In order to motivate our construction we revise in the next section the concept of sine-type functions.

\section{Sine-type functions}
An entire function $F$ of exponential type is called a \emph{sine-type function} (of type $\sigma$) if
\begin{enumerate}
\item $F$ has exponential type $\sigma\;(0<\sigma<\infty)$ in each of the half-planes $\mathbb C_+:=\{\Im z>0\}$ and $\mathbb C_-:=\{\Im z<0\}$,
\item the zeros $\{\lambda_n\}_{n\in\mathbb Z} $ of $F$ are located in some strip $\{|\Im z|\leq h\}$ and satisfy the separation condition (\ref{sep}),
\item for some $y\neq 0, \;C,c>0$ and all $x\in\mathbb R$ holds 
$c\leq |f(x+{\rm i}y)|\leq C$.
\end{enumerate}

This definition goes back to Levin \cite{lev61},  \cite{lev69}.  Together with Golovin \cite{gol} he proved

\begin{thm}
\label{levinthm}
The zeros $\{\lambda_n\}_{n\in\mathbb Z}$ of a function of sine-type $\pi$ form a {\cis}. 
\end{thm}

An exposition of this result and some further developments can be found in \cite{lev96} and \cite{young}. Though not every {\cis} coincides with the zeros of some sine-type function, the introduction of this notion was not too far away from the exhaustive characterization in Theorem \ref{pavlov}, as the following theorem of Avdonin and Jo\'{o} \cite{avdjoo}  shows.

\begin{thm}
If $\{\lambda_n\}_{n\in \mathbb Z} \subset \mathbb R$  is a complete interpolating sequence then there is $d\in (0,\frac{1}{4})$ and a function $F$ of sine-type $\pi$ and with its zeros in $\{\mu_n\}_{n\in\mathbb Z}$ such that
\[d(\lambda_{n-1}-\lambda_n)\leq \mu_n-\lambda_n\leq d(\lambda_{n+1}-\lambda_n)\]
for all $n\in\mathbb Z$. 
\end{thm}

As noticed by Levin and Ostrovskii \cite{levostr1}, a parametric description of all sine-type functions is desirable, and this problem was completely solved by Eremenko and Sodin \cite{eresod}. Since their construction will be of importance to us later, we describe it in some detail. First we recall that Levin and Ostrovskii \cite{levostr2} proved that a sequence $\{\lambda_n\}_{n\in\mathbb Z}$  with $|\Im \lambda_n|\leq h$ is the zero sequence of a sine-type function if and only if  $\{\Re \lambda_n\}_{n\in\mathbb Z}$ is so.  We can therefore restrict our considerations to sine-type functions $F$ with real zeros. Such functions belong to the \emph{Laguerre-P\'{o}lya class} ${\cal LP}$  of real entire functions approximable by real polynomials with real zeros uniformly  on every compact subset of $\mathbb C$.  We may number the critical points of a function $F\in{\cal LP}$  in increasing order and counting multiplicities, so  that the indices of all points of maximum are even. Let $c_n$   be the value of $F$ at the $n$-th critical point and denote by Cr\,$F:=\{\ldots, c_{-1}, c_0, c_{1}, \ldots\}$ the set of all critical values. If Cr\,$F$ is finite from the left (from the right) and $c_m$ is its first (last) member then we set $c_{m-1}=\lim_{x\to -\infty} f(x)$ ($c_{m+1}=\lim_{x\to +\infty} f(x)$), which is allowed to be infinite. The sequence Cr\,$F$ is alternating, i.e. $(-1)^nc_n\geq 0$.  

For an arbitrary alternating sequence $s=\{c_n\}$ a function $F\in {\cal LP}$ with Cr\,$F=s$ can be constructed in the following way. Define a simply connected comb-like domain
\[\Omega(s):=\mathbb C\setminus \bigcup_n \{ z=x+{\rm i}n\pi\colon -\infty<x\leq \log |c_n|\}.\]
If there is an  infinite  first (last) member $c_m$ of the sequence  we consider the domain lying in  the halp plane $\{\Im z>(m-1)\pi\}$ ($\{\Im z<(m+1)\pi\}$). Let $\varphi:\mathbb C_-\to \Omega(s)$ be a conformal map of the lower half-plane onto $\Omega (s)$ such that $\Re \varphi({\rm i}y)\to \infty $ as $y\to-\infty$. Then the function
\begin{equation}
\label{expf}
F(z)={\rm e}^{\varphi (z)}
\end{equation}
can be extended onto $\mathbb C$ by the Schwarz reflection principle and it can be shown that $F\in {\cal LP}$ and Cr\,$F=s$. Conversely, every function $F\in {\cal LP}$ has the representation (\ref{expf}) with $\varphi$ a conformal map of $\mathbb C_-$ onto $\Omega ({\rm Cr}\, F)$.  Eremenko and Sodin \cite{eresod}  specify now conditions on the sequence $\{c_n\}$ in order that the function (\ref{expf}) is of sine type.

\begin{thm}
\label{eresodthm}
A real entire function $F$ is of sine-type if and only if it can be represented in the form {\rm (\ref{expf})}  where $\varphi:\mathbb C_-\to\Omega (s)$ is a conformal map with $\Re \varphi (\i y)\to \infty $ as $y\to-\infty$, and the sequence $s= \{c_n\}_{n\in \mathbb Z}={\rm Cr}\,F$ satisfies 
\begin{equation}
\label{erecond}
c\leq |c_n|\leq C\qquad \forall n\in \mathbb Z
\end{equation}
with some constants $C,c>0$. 
\end{thm}
In fact, \cite{eresod} contains a more general statement that even characterizes sine-type functions where the separability condition (\ref{sep}) in the definition is omitted, but we do not  need this since in view of condition (i) in Theorem \ref{pavlov} we will only be concerned with separated
sequences. Parameterizations of classes of entire functions by conformal mappings onto certain slit  domains have a long history connected with the names of MacLane, Vinberg, Marchenko and Ostrovskii, see the references in \cite{eresod}. 

\section{Characterization of  {\cis}s}
 For any real {\cis} $\{\lambda_n\}_{n\in\mathbb Z}$ the generating function (\ref{gf}) belongs to the Laguerre-P\'{o}lya class ${\cal LP}$ and has therefore a representation (\ref{expf}) with a conformal map $\varphi: \mathbb C_-\to \Omega{(s)}$. It is therefore natural to ask for conditions on the sequence $s=\{c_n\}$ that characterize real {\cis}s. Clearly, such conditions must be more general than (\ref{erecond}).

 The crucial point is the Muckenhoupt condition (\ref{A2}), 
  and in order to develop an idea we remark that any finite interval $I$ can be decomposed into subintervals say $I_p,\ldots, I_q$ such that the segment $I_n+{\rm i}y$ ($y<0$ fixed) is mapped by $\varphi$ near the $n$-th endpoint of the cuts of $\Omega(s)$, i.e.
  \[ \varphi (x+{\rm i}y)\approx \log |c_n|+ \i n\pi,\qquad x\in I_n, n=p,\ldots, q,\]
  see Figure \ref{picture1}.

\begin{figure}[h]
\begin{center}

\begin{tikzpicture}[scale=1.1]
\begin{scope}[>=latex]
%links
\fill[black!15!white]  (-3, 1) rectangle (3,-3) node[pos=0.9, color=black]{\small $\mathbb C_-$};
\draw[->]  (-3, 1) -- (3,1) node[anchor=south]{\scriptsize Re};
\draw(-3,-1)--(3,-1);
\draw[->] (-2.5, -3) -- (-2.5, 3) node[anchor=west]{\scriptsize Im};
%Striche oben
\draw (-2.0, 1.1) -- (-2.0, 0.9) node[anchor=north] {\scriptsize $\lambda_p$};
\draw (-1.0, 1.1) -- (-1.0, 0.9) node[anchor=north] {\scriptsize $\lambda_{p+1}$};
\draw (0.7, 1.1) --(0.7, 0.9) node[anchor=north] {\scriptsize $\lambda_{q-1}$};
\draw (1.5, 1.1) -- (1.5, 0.9)node[anchor=north] {\scriptsize $\lambda_q$};
\draw (2.5, 1.1) -- (2.5, 0.9)node[anchor=north] {\scriptsize $\lambda_{q+1}$};
%dicke Linien
\draw[very thick, red] (-1.8,-1)--(-1.1,-1) node[midway, below]{\scriptsize $I_p+\i y$};
\draw[very thick, red](0.7,-1)--(1.6,-1)node[midway, above]{\scriptsize $I_{q-1}+\i y$};
\draw[very thick, red] (1.6,-1)--(2.6,-1) node[midway, below]{\scriptsize $I_q+\i y$};
%Striche unten
\draw (-1.8, -1.1) -- (-1.8, -0.9);
\draw (-1.1, -1.1) -- (-1.1,- 0.9);
\draw (0.7, -1.1) --(0.7, -0.9);
\draw (1.6, -1.1) -- (1.6, -0.9);
\draw (2.6,- 1.1) -- (2.6, -0.9);

\begin{scope}[xshift=-2cm]
%rechts
\fill[black!15!white] (6,-3) rectangle (12.5, 3) node[color=black,pos=0.9]{\small $\Omega (s)$};
\draw[->] (6,-2.7) -- (12.2,-2.7) node[above]{\scriptsize Re};
\draw[->] (6.4, -3) --(6.4, 3) node[right]{\scriptsize Im}; 

%Kurve 
\begin{scope}[xshift=8cm, yshift=5.3cm];
\begin{scope}[rotate=-90, scale=1.6 ]
\draw[smooth] plot coordinates {(1.4374351013657951 ,  0.51401549989056272)  (1.5361863196469436,   0.59253684942805074)
(   1.6167524942955553  , 0.67935556334084091)
  ( 1.6791287100580834 ,  0.765906524610951)
 (  1.7272531648401537 ,  0.846936880518840)
 (  1.7652816992602585,  0.920546262747710)
 (  1.7963100557161231 ,  0.986526040303256)
  ( 1.822435937392788 ,  1.0452777382910370)};
\draw[very thick, red, smooth] plot coordinates {(  1.8224359373927888 , 1.045277738291037)
  ( 1.845067364214854  , 1.097346754059864)
  ( 1.8651708256655721 ,  1.14325299198707)
  ( 1.883433644192128  , 1.1834387204275574)
  ( 1.900365863295388,  1.218258751027608)
   (1.9163650058853743 ,  1.2479842282325972)
 (  1.9317590019288833  , 1.2728093514136121)
  ( 1.9468367081187159 ,  1.2928571729365308)
  ( 1.9618720089807882  , 1.3081831557488677)
  ( 1.9771457523581129 , 1.3187761015922608)
  ( 1.9929692080605363 ,  1.324556430488256)
  ( 2.0097131634965968 ,  1.325372110019501)
  ( 2.0278483321476379,  1.320993251082972)
  ( 2.0480059712454692 ,  1.311108336564063)
  ( 2.0710734362123548  , 1.2953302660520269)
  ( 2.0983488607122336 ,  1.2732341996809775)
  ( 2.1317902761968122  , 1.244485440066544)
  ( 2.1743861998566247  , 1.2092076774483023)
  ( 2.2305380213440076 ,  1.168943909975855)
  ( 2.3056859001422230 , 1.12876898756593)
  ( 2.4027381914586181 , 1.09998671003149)
  ( 2.5138945493189091 , 1.0973439554801918)
  ( 2.6197768563026056 ,  1.1251394663146650)
  ( 2.7058146030667132 , 1.1716934338159)
  ( 2.7707945031629020 , 1.2229021669995115)
  ( 2.8197184005709084 , 1.270870065295753)
  ( 2.8576722719744074 ,  1.3126538653726061)
  ( 2.8883213008244777  , 1.3475785665908060)
  ( 2.9141031875415639 , 1.3757803847935428)
  ( 2.9366416112877434   ,1.3976099888305717)
  ( 2.9570577717301516,  1.4134133472354693)
  ( 2.9761716990760361 , 1.4234579381325327)
  ( 2.9946311833684272 , 1.4279113752369150)
  ( 3.0129997427210995 ,  1.4268370423954808)
  ( 3.0318250687805426 , 1.420193670177026)
  ( 3.0517036835413238 ,  1.407834874893236)
  ( 3.0733563775502648  , 1.3895097438808006)
  ( 3.0977322248341177,  1.3648717276541389)
  ( 3.1261668498161725  , 1.3335166659111459)
  ( 3.1606310735923078 ,  1.295105425217764)
  ( 3.2040996757699496  , 1.2497145565453867)
  ( 3.2609456433016213  , 1.1987538452309885)
  ( 3.3366500137075157 ,  1.1470092944274175)
  ( 3.4344827752072558  , 1.1053880963383291)};
    \node[color=red] at (2.4,1.8) {\scriptsize $\varphi (I_q+\i y)$}; 
 \node[color=red] at (3.3,1.9) {\scriptsize $\varphi (I_{q-1}+\i y)$}; 
\draw[smooth] plot coordinates    {(3.4344827752072558 ,  1.1053880963383291)
 (  3.5474437722932288  ,1.088660888980066)
 (  3.656534530021926 , 1.101925897452975)
 (  3.7466406879486325  , 1.1339849024485991)
 (  3.8160989494576536  , 1.1704044836609020)
 (  3.869962025323492  , 1.2027931647415819)
 (  3.913617272011930  , 1.2278267150113289)
 (  3.951145495476087  , 1.2445166804492545)
 (  3.9855291432408220 , 1.25268686436115)
 (  4.0191518846986440  , 1.25235090889164)
 (  4.0542561925979719  , 1.243515267209409)
 (  4.0933774211171539  , 1.226216844496429)
 (  4.1398217229899688 , 1.200832965359049)
 (  4.1981392369695962 , 1.1689494049636731)
 (  4.273967963000052  , 1.1353290279231358)
 (  4.371029450986588 , 1.110714639362754)
 (  4.483238599733396  , 1.1101935504547709)
 (  4.5919845516442717  , 1.1401305623217306)
 (  4.6813588573329437  , 1.1902868305141232)};
\draw[very thick, red, smooth] plot coordinates   {( 4.6813588573329437 ,  1.19028683051412)
(   4.7488977589371046 , 1.246259604261888)
 (  4.7993367359493302  , 1.2995344144148990)
  ( 4.83789212837148  , 1.3468782415846468)
  ( 4.868381714709805  , 1.38757410514576)
  ( 4.893333685017062 , 1.421831387682976)
  ( 4.91440069532537  , 1.450124965187161)
  ( 4.932681820238216  , 1.47295323837021)
  ( 4.948928599519156  , 1.49075940028944)
  ( 4.963671304192025  , 1.5039127044044784)
  ( 4.9772964956965264  , 1.5127103680226970)
  ( 4.990095615241791  , 1.5173854195879195)
  ( 5.002296275000851  , 1.51811516651688)
  ( 5.0140830836352368  , 1.5150285057997)
  ( 5.02561206256903  , 1.5082116164628)
  ( 5.037021130085438  , 1.497712029237914)
  ( 5.048438232512948  , 1.4835411923116284)
  ( 5.059988199265982  , 1.465675648863278)};
\node[color=red] at (4.8,2.0)  {\scriptsize $\varphi (I_{p}+\i y)$}; 
\draw[smooth] plot coordinates {(   5.059988199265982 ,  1.465675648863278)
(   5.071799138009604,   1.4440568925861370)
(   5.0840090915330052,   1.4185899026268498)
(   5.0967737218939657,   1.3891402926445851)
(   5.1102759820127268,   1.3555299512710834)
(   5.124739135566241,   1.3175310265720872)};
\draw[smooth] (2.8, 1.2)--(2.75,1.25);
\end{scope}
\end{scope}

%schlitze
\draw[thick] (6, -2.7)--(9.6, -2.7) node[anchor=north east]{\scriptsize $\log |c_p|+p\i \pi$};
\draw[thick] (6, -1.1)--(9.1, -1.1) node[anchor=north east ]{\scriptsize $\log |c_{p+1}|+(p+1)\i \pi$};
\draw[thick] (6, 0.5)--(9.5, 0.5) node[anchor=north east]{\scriptsize $\log |c_{q-1}|+(q-1)\i \pi$};
\draw[thick] (6, 2.1)--(9.6, 2.1) node[anchor=north east]{\scriptsize $\log |c_q|+q\i \pi$};
\end{scope}

%Pfeil
\draw[->, thick]  (2.8, -2) parabola bend ( 3.5, -2.4)  (4.2, -2);
\node at (3.5  ,-2){\small $\varphi$} ; 
\end{scope}

\end{tikzpicture}

\caption{Conformal mapping of $\mathbb C_-$ onto $\Omega(s)$}
\label{picture1}
\end{center}
\end{figure}
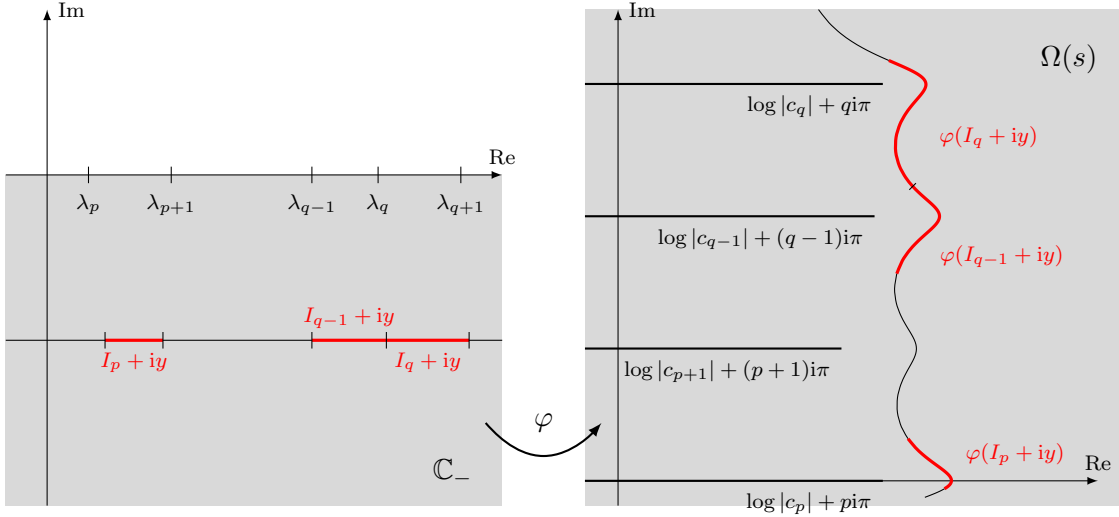

  If we approximate the occurring integrals by its Riemann sums we get
  \[\int_I |F(x+\i y)|^2dx \approx \sum_{n=p}^q c_n^2 |I_n|, \quad \int_I |F(x+\i y)|^{-2}dx \approx \sum_{n=p}^q c_n^{-2} |I_n|.\]
  The properties (\ref{sep}), (\ref{eq2}) let us suppose that also
  \[\delta'\leq |I_n|\leq \Delta', \qquad  n=p,\ldots, q,\]
  for some constants $\delta', \Delta'>0$ that do not depend on $I$. Hence it turns out that it will just be the discrete Muckenhoupt condition (\ref{A2t}) for the numbers $d_n=c_n^2$  that is our desired characterization.  Making these informal arguments rigorous is enough for showing the following statement. 
  
\begin{thm}
\label{charthm} 
Let $s=\{c_n\}_{n\in\mathbb Z}$  be a sequence with $(-1)^n c_n \geq 0$, and $\{d_n\}_{n\in \mathbb Z}= 
\{c_n^2\}_{n\in\mathbb Z}$ satisfy the discrete Muckenhoupt condition {\rm  (\ref{A2t})}. Then for  every conformal map $\varphi: \mathbb C_-\to \Omega (s)$ with $\lim_{y\to - \infty}\Re \varphi(\i y)=\infty$ the function $F$ in {\rm (\ref{expf})} is an entire function of exponential type. $\varphi$ can be normalized so that the exponential type of $F$ is $\pi$, and in this case the zero set  $\{\lambda_n\}_{n\in \mathbb Z}\subset \mathbb R$  of $F$ is a {\cis}. Conversely, every {\cis} is obtained in this way. 
\end{thm}

\begin{proof}
It suffices to adapt arguments from \cite{eresod}, which make use of  ideas from \cite{gold}. For the reader's convenience we provide the details.  In order to structure the proof we formulate important steps  as lemmas. 

First suppose that  the sequence $\{d_n\}_{n\in \mathbb Z}= 
\{c_n^2\}_{n\in\mathbb Z}$ satisfies the discrete Muckenhoupt condition (\ref{A2t}), and $\varphi: \mathbb C_-\to \Omega (s)$ is a conformal map  with $\lim_{y\to - \infty}\Re \varphi(\i y)=\infty$.   Choosing $I=\{p, p+1, \ldots, q\}$ in (\ref{A2t}) we obtain
\[c_p^2c_q^{-2}\leq \sum_{n=p}^qc_n^2\sum_{n=p}^qc_n^{-2}\leq C(q-p+1)^2.\]
Hence
for some $C_1>0$  we have 
\begin{equation}
\label{ersteugl}
|\log|c_p|-\log |c_q| | \leq C_1+\log |p-q|\qquad \forall p\neq q,
\end{equation}
and in particular
\begin{eqnarray}
\label{ugl}
-C_2-\log |n|\leq \log |c_n|\leq  C_2+\log |n| \qquad  \forall n\in\mathbb Z\setminus\{0\},
\end{eqnarray}
where $C_2:=\log |c_0|+C_1$. For $r>0$ let $L(r)$ be the connected arc of $\Omega(s)\cap\{|z|=r\}$ that intersects the positive real axis and let $l(r)$ denote its length.  Next we show that the preimage of $L(r)$ is almost a  semicircle in $\mathbb C_-$. 
\begin{lem}
\label{gwlemma}
 For $r\to\infty$ the relations
  \begin{equation}
\label{urbild}
\sup_{z\in L(r)}|\varphi^{-1}(z)|=A(1+o(1))r\qquad \mbox{and}\qquad \inf_{z\in L(r)}|\varphi^{-1}(z)|=A(1+o(1))r
\end{equation}
hold with a constant $A>0$.
\end{lem} 
\begin{proof} Let $\Omega(r_1, r_2)\;(0<r_1<r_2)$ be  the component  of $\Omega(s)\setminus (L(r_1)\cup L(r_2))$ which contains $\frac{1}{2}(r_1+r_2)$ and $M(r_1, r_2):=\mod \Gamma_{r_1, r_2}$ be the module of the family $\Gamma_{r_1, r_2}$ of curves in $\Omega(r_1, r_2)$ separating $L(r_1)$ and $L(r_2)$. 
The paper \cite{gold} provides the following consequence of one of Teichm\"uller's module theorems: {\it If $\Phi(r),\; r\geq r_0$ is a function such that 
\begin{equation}
\label{goldbg}
\lim_{r_1, r_2\to\infty} M(r_1, r_2)-(\Phi(r_2)-\Phi(r_1))=0
\end{equation}
then there is $A>0$ such that}
\begin{equation*}
\sup_{z\in L(r)}|\varphi^{-1}(z)|=(1+o(1))A{\rm e}^{\pi\Phi(r)}\qquad \mbox{\it and}\qquad 
\inf_{z\in L(r)}|\varphi^{-1}(z)|=(1+o(1))A{\rm e}^{\pi\Phi(r)}. 
\end{equation*}
Applying this theorem for $\Phi(r)=\frac{1}{\pi}\log r$ gives the assertion.

Let $a(r), b(r)$ be the endpoints of $L(r)$. Then $a(r)=x(r)+\i\pi n(r)$ for some $n=n(r)>0$, and  $x=x(r)\leq \log |c_n|$. If  $x(r)\leq 0$ then $\log |c_{n+1}|+\i\pi (n+1)$ lies outside the circle of radius $r$, hence 
\[(\log|c_{n+1}|)^2+\pi^2 (n+1)^2>r^2=x^2+\pi^2 n^2\]
 and (\ref{ugl}) yields
\[ (C_2+\log (n+1))^2+(2n+1)\pi^2>x^2,\]
 and therefore  with $r\geq \pi n$
\[-x(r)\leq C_3\sqrt{n}\leq \pi^{-\frac{1}{2}}C_3\sqrt{r}\]
with some constant $C_3>0$ independent of $n$ and $r$.
For $x(r)\geq 0$ we apply (\ref{ugl}) to obtain
\begin{equation}
\label{augl}
x(r) \leq \log  |c_n|\leq C_2+\log n\leq C_2-\log \pi+\log r.
\end{equation}
The analogous estimate
\begin{equation}
\label{bugl}
-\pi^{-\frac{1}{2}} C_3\sqrt{r}\leq \Re b(r)\leq C_2-\log \pi+\log r
\end{equation}
is valid for the other endpoint. Hence we can find for the length of $L(r)$
\[\pi r-2r\arcsin \left(\frac{C_2-\log \pi +\log r}{r}\right ) \leq l(r)\leq \pi r+2r\arcsin \left(\frac{\pi^{-\frac{1}{2}}C_3 \sqrt{r}}{r}\right ),\]
and therefore
\begin{equation}
\label{growth}
\pi(r-C_4\log r)\leq l(r)\leq \pi (r+ C_4\sqrt{r})
\end{equation}
 with a constant $C_4>0$. 
By the elementary properties of  the module (cf. \cite{ahl}), $M(r_1, r_2)$ is bounded from above by the module $s(r_1, r_2)$  of  the family of curves joining the  noncircular sides of 
\begin{equation}
\label{s}
S(r_1, r_2):=\{ z\in \mathbb C\colon r_1< |z|<r_2, |\arg z|<\frac{\pi}{2}(r- C_4\log r)/r\}, 
\end{equation}
i.e.
\begin{equation}
\label{a}
M(r_1, r_2)\leq s(r_1, r_2).
\end{equation}
We postpone  the proof that 
\begin{equation}
\label{b}
\left| s(r_1, r_2)-\frac{1}{\pi}\log\frac{r_2}{r_1}\right |\to 0 \qquad \mbox{as} \qquad r_1, r_2\to \infty
\end{equation}
to Lemma \ref{modullemma}. A lower estimate of the module follows from the Gr\"otzsch principle (cf. \cite{pom}, 
Proposition 11.12)
\[ M(r_1, r_2)\geq \int_{r_1}^{r_2}\frac{dr}{l(r)}.\]
With the help of (\ref{growth}) we find 
\begin{equation}
\label{c}
M(r_1, r_2) -\frac{1}{\pi}\log\frac{r_2}{r_1}
\geq\frac{1}{\pi}\int_{r_1}^{r_2} \frac{dr}{r+C_4\sqrt{r}}-\frac{1}{\pi}\int_{r_1}^{r_2}\frac{dr}{r}
=-\frac{1}{\pi}\int_{r_1}^{r_2} \frac{C_4dr}{(r+C_4\sqrt{r})\sqrt{r}}
\end{equation}
In view of $\displaystyle \int_{1}^{\infty} \frac{dr}{(r+C_4\sqrt{r})\sqrt{r}}<\infty$ the right hand side of (\ref{c}) tends to $0$ as $r_1, r_2\to \infty$. Together with (\ref{a}), (\ref{b}) this implies (\ref{goldbg}) and the lemma is proved. 
\end{proof}

From (\ref{urbild}) follows  that for every sufficiently large $R>0$ the preimage $\varphi^{-1}(L(r))$ for $r=\frac{2}{A}R$  and an interval of the real axis are the boundary of a domain containing $\mathbb C_-\cap\{|z|=R\} $ in its interior. Thus $\mathbb C_-\cap\{|z|=R\} $ is mapped by $\varphi$ into the component of $\Omega(s)\setminus L(r)$ not containing $r+1$. 
Consequently, for every $z\in \mathbb C_-$ with $|z|=R$ either
\[\Re\varphi (z)\leq |\varphi(z)|\leq  r =\frac{2}{A}R=\frac{2}{A}|z|\]
or (by (\ref{augl}) and (\ref{bugl}))
\[\Re \varphi (z)\leq C_2-\log \pi+  \log r=C_5+ \log R=C_5+\log |z|,\]
and we have proved
 \[ |F(z)|=O\left ( {\rm e}^{ C_6|z|}\right ), \qquad C_6:=\max\left (\frac{2}{A},1\right ), \]
i.e. $F(z)$ is  of exponential type in $\mathbb C_-$. Since $F$ is continuous in $\overline{\mathbb C}_-$ and real-valued on the real axis, it can be extended to $\mathbb C$ by the reflection principle and is of exponential type there. 

Note that  for $a>0$ also $\varphi(az)$ is a conformal map of $\mathbb C_-$ onto $\Omega(s)$.  Choosing  $a>0$ appropriately     we can always achieve that the exponential type of $F$ is equal to  $\pi$, and we assume this normalization for the rest of the proof for $a=1$.  
Then we have to show that $\{\lambda_n\}_{n\in\mathbb Z}=F^{-1}(0)$ is a \cis. 

\begin{lem} 
The sequence $\{\lambda_n\}_{n\in\mathbb Z}$ is separated.  
\end{lem}

\begin{proof}
For $\zeta\in \Omega(s)$ let $L_\zeta(r)$ be the connected arc of $\{|\zeta-z|=r\}\cap \Omega(s)$ that contains $\zeta+r$ and $\Omega_\zeta(r)$ the component of $\Omega(s)\setminus L_\zeta(r)$ that contains $\zeta+r/2$.   

Denote by $P_t, \,t\in \mathbb R$  the polygonal line
\[P_t:=\bigcup_{n\in\mathbb Z} \{x(\log |c_n|+\i n\pi)+ (1-x)(\log |c_{n+1}| +\i (n+1)\pi) +t\colon x\in [0,1]\}.\]
$P_0$ just connects the endpoints of the slits of $\Omega(s)$. For fixed $t>0$ and every $\zeta \in P_t$
 the length $l_\zeta(r)$ of $L_\zeta(r)$ satisfies  the estimate (\ref{growth}) with a constant $C_4>0$ that does not  depend on $\zeta$ (but, of course, on $t$). Hence we conclude as in Lemma \ref{gwlemma}
 \[A(1+g_1(\zeta,r))r\leq \inf_{z\in L_\zeta(r)}|\varphi^{-1}(z)|\leq \sup_{z\in L_\zeta(r)}|\varphi^{-1}(z)| \leq 
A(1+g_2(\zeta,r))r, \qquad r\geq r_0(\zeta),\]
where $g_1, g_2$ are such that
\[\lim_{r\to \infty} g_j(\zeta,r)=0,\qquad j=1,2,\, \zeta\in P_t.\] 
According  to \cite{mar}, \cite{tsui}, the harmonic measure $\omega_\zeta(r)$ of the boundary arc $L_\zeta(r)$  at the point $\zeta$ with respect to  the domain $\Omega_\zeta(r)$ allows the estimate
\[\omega_\zeta(r)\leq C_7\exp\left(-\pi \int\limits_{r_\zeta}^r\frac{d\varrho}{l_\zeta(\varrho)}\right)  \leq C_7\exp\left(-\int\limits_{r_\zeta}^r\frac{d\varrho}{\varrho+C_4\sqrt{\varrho} }\right) \leq \frac{C_8}{r}, \]
where $r_\zeta={\rm dist}\, (\zeta, \partial \Omega (s))>0$. For the harmonic function $-\Im \varphi^{-1}$ we find now
\[-\Im \varphi^{-1}(\zeta)\leq  \omega_\zeta (r)\sup_{z\in L_\zeta(r)}|\varphi^{-1}(z)|\leq \frac{C_8}{r}A(1+g_2(\zeta,r))r.\]
Letting $r\to\infty$ we conclude 
\begin{equation}
\label{streifen}
0>\Im \varphi^{-1}(\zeta)\geq -AC_8, \qquad \zeta \in P_t,
\end{equation}
i.e. the ''left'' component of $\Omega(s)\setminus P_t$ is mapped into some strip below the real axis. 

Let $x_n:=\varphi^{-1}(\log |c_n|+\i n\pi)$ be the critical points of $F$  and consider the non-Euclidean segments  $H_{n,m}\subset\overline{\mathbb C}_-$ connecting $x_n$ and $x_m$, i.e. $H_{n,m}$ are Euclidean semicircles with centers $(x_n+x_m)/2$. $\varphi(H_{n,m})$ is a curve connecting the endpoints of the $n$-th and $m$-th slit of $\Omega(s)$, and according to the Gehring-Hayman Theorem (see \cite{pom}, Theorem 4.20) it is not much longer than any curve $\gamma_{n,m}\subset \Omega(s)$ connecting the same endpoints, i.e.
\[{\rm length}\,(\varphi(H_{n,m}))\leq C_9\,{\rm length}\, (\gamma_{n,m}),\]
where $C_9>0$ is a universal constant.  Choosing $\gamma_{n, n+1}$ as the straight line segment between two successive slit endpoints and noting that by (\ref{ersteugl}) this is not longer than $\sqrt{\pi^2+C_1^2}$, we find that
\begin{equation}
\label{laenge}
{\rm length}(\varphi(H_{n,n+1}))\leq C_9\sqrt{\pi^2+ C_1^2}, \qquad n\in\mathbb Z,
\end{equation}
and thus $\varphi (H_{n, n+1})\cap P_t=\emptyset$ for all $n\in\mathbb Z$ provided that $t>0$ has been chosen sufficiently large. From (\ref{streifen}) follows that
\[H_{n,n+1}\subset \{z\in\mathbb C\colon -AC_8\leq \Im z\leq 0\}, \qquad n\in\mathbb Z,\]
and hence the (Euclidean) radii of the semicircles $H_{n,n+1}$ are bounded. We infer
\begin{equation}
\label{sep3}
\Delta_1:=\sup_{n\in\mathbb Z}|x_{n+1}-x_n |<\infty.
\end{equation}
Let $G$ be the auxiliary domain
\[G:=\{x+\i y\in\mathbb C\colon x>\max(\log |y|, C_{10})\},\]
where $C_{10}<0$ (sic!) is chosen so small that $G$ has the property that for every $n\in\mathbb Z$ there is $\zeta_n\in\mathbb C$ such that $G+\zeta_n\subset \Omega(s)$ and $C_{10}+1+\zeta_n$ is ''left'' of $P_{t_0}$ for $t_0=- C_9\sqrt{\pi^2+ C_1^2}$.  Let further $\psi:\mathbb C_-\to G$ be a conformal mapping with 
\[\psi (\infty)=\infty, \qquad \lim_{|z|\to\infty} \frac{|\psi(z)|}{|z|}=A,\]
and let $y_0<0$ be so close to $0$  that $C_{10}+1\in \psi(\mathbb R +\i y_0)$. Then $\zeta_n+\psi(z)$ maps $\mathbb C_-$ onto $G+\zeta_n$, and according to  the Lindel\"of principle (cf. \cite{lav}) $\varphi(\mathbb R+\i y_0)$ lies in the ''left'' component of $\Omega(s)\setminus (\psi(\mathbb R+\i y_0)+\zeta_n)$, see Figure \ref{figure3}.

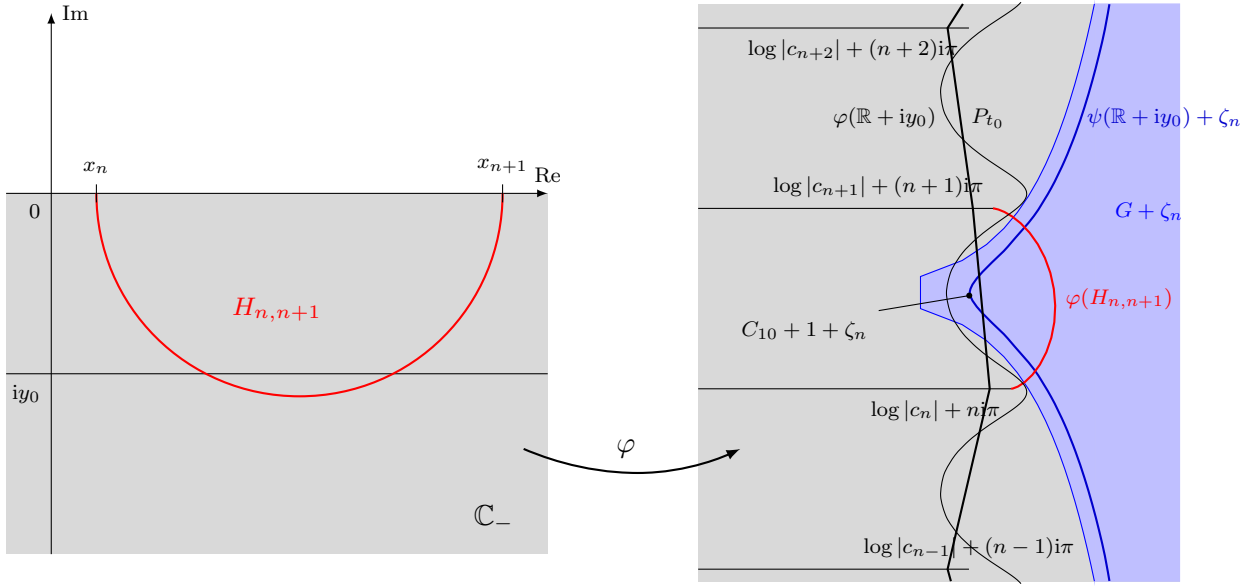
\begin{figure}[h]
\begin{center}

\begin{tikzpicture}[scale=0.8,>=latex]
%links
\begin{scope}[scale=1.5, yshift=0.5cm]
\fill[black!15!white]  (-3, 1) rectangle (3,-3) node[pos=0.9, color=black]{\small $\mathbb C_-$};
\draw[->]  (-3, 1) -- (3,1) node[anchor=south,]{\scriptsize Re};
\draw(-3,-1)--(3,-1);
\draw[->] (-2.5, -3) --(-2.5,-1.0) node[anchor=north east]{\scriptsize $\i y_0$}--(-2.5,1)node[anchor=north east]{\scriptsize $0$} --(-2.5, 3) node[anchor=west]{\scriptsize Im};

%Halbkreis
\draw[color=red, thick] (-2.0, 1.0) arc (180:360:2.25cm) node[anchor=north, midway]{\scriptsize $H_{n,n+1}$};
%Striche oben
\draw (-2.0, 1.1) node[anchor=south] {\scriptsize $x_{n}$}-- (-2.0, 0.9) ;
\draw (2.5, 1.1)node[anchor=south] {\scriptsize $x_{n+1}$}--(2.5, 0.9) ;%dicke Linien

\end{scope}

%rechts
\begin{scope}[xshift=7cm]

\fill[ color=black!15!white] (0, -4.2) rectangle (8, 5.4);

%Schnitte
\draw (0,5)--(4.5 ,5) node[anchor=north east]{\scriptsize $\log |c_{n+2}|+(n+2) \i \pi$} ;
\draw (0, 2)--(4.92,2) node[anchor=south east]{\scriptsize $\log |c_{n+1}|+(n+1) \i \pi$} ;
\draw (0,-1.0) -- (5.2,-1.0) node[anchor=north east]{\scriptsize $\log |c_n | + n\i\pi$} ;
\draw (0,-4)--(4.5,-4) node[anchor=south]{\scriptsize $\log |c_{n-1}|+(n-1) \i \pi$} ;

%G
\filldraw[color=blue!25!white] (6.6,-4.2)rectangle (8, 5.4) node[pos=0.64, color=blue]{\scriptsize $G+\zeta_n$};
\begin{scope}[xshift=-5.9cm, yshift=8.6cm,xscale=5,yscale=2]
\filldraw[color=blue!25!white, rotate=-90] plot coordinates {(   1.6000000 ,  2.4969813)
(   1.7333333   ,2.4855496)
(  1.8666667  ,2.4734247)
 (  2.0000000 , 2.460517)
 (  2.133333  , 2.446718)
 (  2.2666667  , 2.4318968)
 (  2.400000  , 2.4158883)
 (  2.5333333  , 2.3984860)
  ( 2.6666667  , 2.37942)
  ( 2.8000000  , 2.3583519)
  ( 2.93333  , 2.33479)
  ( 3.0666667  , 2.3080890)
  ( 3.200000  , 2.2772589)
  ( 3.333333  , 2.240794)
  ( 3.4666667  , 2.1961659)
  ( 3.6000000  , 2.1386294)
  ( 3.7333333 , 2.0575364)
  ( 3.866666  , 1.91890)
  ( 4.133333  , 1.9189070)
  ( 4.2666667 , 2.0575364)
  ( 4.4000000 , 2.138629)
  ( 4.533333  , 2.1961659)
  ( 4.6666667  , 2.2407946)
  ( 4.8000000  , 2.277258)
  ( 4.9333333  , 2.30808)
  ( 5.0666667  , 2.3347953)
  ( 5.2000000  , 2.3583519)
  ( 5.3333333  , 2.3794240)
  ( 5.4666667  , 2.3984860)
  ( 5.6000000  , 2.4158883)
  ( 5.7333333  , 2.4318968)
  ( 5.8666667  , 2.4467184)
  ( 6.0000000  , 2.4605170)
  ( 6.1333333  , 2.4734247)
  ( 6.2666667  , 2.4855496)
  ( 6.4000000  , 2.4969813)}--cycle;
\draw[color=blue, rotate=-90] plot coordinates {(   1.6000000 ,  2.4969813)
(   1.733333   ,2.4855496)
(  1.8666667   ,2.4734247)
 (  2.0000000  , 2.4605170)
 (  2.1333333  , 2.4467184)
 (  2.266666  , 2.4318968)
 (  2.4000000  , 2.4158883)
 (  2.533333 , 2.3984860)
  ( 2.6666667  , 2.3794240)
  ( 2.8000000  , 2.3583519)
  ( 2.9333333  , 2.3347953)
  ( 3.0666667  , 2.3080890)
  ( 3.200000  , 2.2772589)
  ( 3.333333  , 2.2407946)
  ( 3.466666  , 2.1961659)
  ( 3.6000000  , 2.1386294)
  ( 3.733333  , 2.0575364)
  ( 3.8666667  , 1.9189070)
  ( 4.1333333  , 1.9189070)
  ( 4.2666667  , 2.0575364)
  ( 4.4000000  , 2.1386294)
  ( 4.5333333  , 2.196165)
  ( 4.6666667  , 2.2407946)
  ( 4.8000000  , 2.2772589)
  ( 4.9333333  , 2.3080890)
  ( 5.0666667  , 2.3347953)
  ( 5.2000000  , 2.3583519)
  ( 5.33333  , 2.379424)
  ( 5.466666  , 2.3984860)
  ( 5.6000000  , 2.4158883)
  ( 5.7333333 , 2.4318968)
  ( 5.8666667  , 2.4467184)
  ( 6.0000000  , 2.4605170)
  ( 6.1333333 , 2.4734247)
  ( 6.2666667 , 2.485549)
  ( 6.4000000 , 2.4969813)};
  \draw[color=blue!80!black, thick, rotate=-90, smooth] plot coordinates {(   1.6021557 ,  2.5468290)
(   1.810421 ,  2.5317684)
(   2.0238022 ,  2.5122372)
(   2.2399716 ,  2.4897689)
(   2.4583212,  2.4639642)
(   2.6788999,  2.4339345)
(   2.902218 ,  2.398100)
(   3.1295562 ,  2.3537612)
(   3.364355 ,  2.2954422)
(   3.6186560 ,  2.2089545)
(   4.0000000,  2.0828016)
(   4.3813440,  2.2089545)
(   4.635644 ,  2.2954422)
(   4.8704438 ,  2.3537612)
(   5.0977813 ,  2.3981008)
(   5.3211001,  2.4339345)
(   5.541678 ,  2.4639642)
(   5.760028,  2.4897689)
(   5.976197 ,  2.5122372)
(   6.1895787 ,  2.5317684)
(   6.3978443 ,  2.5468290)};
\end{scope}

%P_t
\draw[thick, xshift=2mm] (4.2, 5.4)--(3.945,5)-- (4.365,2)node[midway, anchor=west]{\scriptsize $P_{t_0}$} --(4.645,-1) -- (3.945, -4) --(4.0, -4.2);

\fill[color=black] (4.505, 0.55)circle (0.5mm);
\draw (4.505, 0.55)-- (3, 0.3) node[anchor=north east]{\scriptsize $C_{10}+1+\zeta_n$};

%phi-Bild
\begin{scope}[xshift=-5cm, yshift=2cm]
\draw[color=red, rotate=-90, scale=3, thick] plot  coordinates {(   1.0000000   ,3.4000000)
(   0.99838738,   3.4099661)
(   0.98794117,   3.4372376)
(   0.96311202,   3.4756173)
(   0.92218468,   3.5183347)
(   0.86600058,   3.5597426)
 (  0.79669486,   3.5956129)
 (  0.71700878,   3.6229170)
 (  0.62999493,   3.6395932)
 (  0.53888494,   3.6444151)
 (  0.44699731,   3.636941)
 (  0.35763655,   3.617509)
 (  0.27396708,   3.587239)
 (  0.19885652,   3.548027)
 (  0.13468829,   3.5025195)
 (  0.083147431,   3.454020)
 (  0.044993619,   3.4063385)
 (  0.019852559,   3.363524)
 (  0.0060805966,   3.329495)
 (  0.00077573911,   3.307572)
 (  0.0000000,   3.3000000)};
\end{scope}
\node[color=red] at (7,0.5) {\scriptsize $\varphi(H_{n,n+1})$};
 
 \begin{scope}[xscale=5, yscale=3.3, xshift=-0.24cm, yshift=3.68cm]
 \draw[rotate=-90, smooth] plot coordinates {(   2.0349706 ,  1.312987)
(   2.0534510,   1.3006848)
(   2.0745038 ,  1.2828744)
(   2.0992145 ,  1.2591144)
(   2.1292082 ,  1.2289823)
(  2.1669667  , 1.1923100)
  ( 2.2162398 ,  1.1498254)
  ( 2.2821768  , 1.1047002)
  ( 2.3695583  , 1.0652831)
  ( 2.4762760  , 1.0458907)
  ( 2.5868509  , 1.0572352)
  ( 2.6824230   ,1.0937063)
  ( 2.7562724   ,1.1404343)
  ( 2.811745  , 1.1865771)
  ( 2.854289  , 1.2273612)
  ( 2.8882360  ,1.2613301)
  ( 2.916540  , 1.2883313)
  ( 2.9411835  , 1.3086131)
  ( 2.9635379  , 1.3224788)
  ( 2.9846235  , 1.3301630)
  ( 3.0052747  , 1.3317908)
  ( 3.0262635  , 1.3273666)
  ( 3.0484056  , 1.3167725)
  ( 3.0726784  , 1.2997777)
  ( 3.100378  , 1.2760707)
  ( 3.1333550  , 1.2453605)
  ( 3.1743622  , 1.2076649)
  ( 3.227471   ,1.1640743)
  ( 3.2980472 , 1.1185279)
  ( 3.3903144  , 1.0806255)
  ( 3.500000   ,1.0651929)
  ( 3.6096856 , 1.0806255)
  ( 3.7019528  , 1.1185279)
  ( 3.772528  , 1.1640743)
  ( 3.8256378   ,1.2076649)
  ( 3.8666450 , 1.245360)
  ( 3.8996219  , 1.2760707)
  ( 3.9273216  , 1.299777)
  ( 3.9515944   ,1.3167725)
  ( 3.9737365   ,1.3273666)
  ( 3.9947253  , 1.3317908)
  ( 4.015376  ,1.3301630)
  ( 4.0364621 , 1.3224788)
  ( 4.0588165  , 1.3086131)
  ( 4.0834595  , 1.2883313)
  ( 4.111764   ,1.2613301)
  ( 4.1457101 ,  1.2273612)
  ( 4.1882541 ,  1.1865771)
  ( 4.243727 ,  1.1404343)
  ( 4.3175770 ,  1.093706)
  ( 4.413149 ,  1.0572352)
  ( 4.5237240 ,  1.0458907)
  ( 4.6304417 ,  1.0652831)
  ( 4.7178232 ,  1.1047002)
  ( 4.7837602 ,  1.1498253)
  ( 4.8330333 ,  1.1923100)
  ( 4.8707918 ,  1.2289823)
  ( 4.9007855 ,  1.2591144)
  ( 4.9254962 ,  1.28287)
  ( 4.9465490 ,  1.3006848)
   (4.9650294,  1.312987)};
\end{scope}
\node at ( 3.1,3.5) {\scriptsize $\varphi(\mathbb R+\i y_0)$};
\node[color=blue!80!black] at ( 7.75,3.5) {\scriptsize $\psi(\mathbb R+\i y_0)+\zeta_n$};

\end{scope}

%Pfeil
\draw[->, thick]  (4.1, -2) parabola bend ( 6, -2.4)  (7.7, -2);
\node at (5.8  ,-2){\small $\varphi$} ; \end{tikzpicture}

\caption{The auxiliary domain $G+\zeta_n$}
\label{figure3} 
\end{center}
\end{figure}

Since $\varphi(H_{n,n+1}) $ lies ''right'' of $P_{t_0}$, the curves $\varphi (\mathbb R+\i y_0)$ and $\varphi (H_{n,n+1})$ have a non-empty intersection.  Hence also $(\mathbb R+\i y_0)\cap H_{n, n+1}\neq \emptyset$ for all $n\in\mathbb Z$  and thus the (Euclidean) radius of $H_{n,n+1}$ cannot be smaller than $|y_0|$. We infer
\begin{equation}
\label{sep1}
\delta_1:=\inf_{n\in\mathbb Z} |x_n-x_{n+1}|>0.
\end{equation}
Let the zero set $\{\lambda_n\}_{n\in\mathbb Z}=F^{-1}(0)=\varphi^{-1}(-\infty)$ be numbered such that $x_n<\lambda_n<x_{n+1}$ for all $n\in \mathbb Z$.  Let $E_n$ be the set $E_n:=\{x+\i n\pi  \colon x\leq \log |c_n|-2C_1\}\cup\{x+\i (n+1)\pi\colon x\leq \log |x_{n+1}|- 2C_1\}$ and let $\tilde{E}_n\subset (x_n, x_{n+1})$ be such that $\varphi (\tilde{E}_n)=E_n$. 
We look at the family $\Gamma_1^{(n)}$ of curves connecting $\tilde{E}_n$ with $[x_{n+1}, \infty )$ in $\mathbb C_-$ and at the family   $\Gamma_2^{(n)}$ of curves connecting $\tilde{E}_n$ with $(-\infty , x_{n}]$ in $\mathbb C_-$. 
Further let
\begin{eqnarray*}
E_n^\pm&:=&\bigcup _{k=n+1}^\infty\{x+\i k\pi\colon x\leq \log |c_n|\pm C_1\pm \log |k-n|\},\\
\Omega_n^\pm&:=&\mathbb C\setminus \bigcup_{k\in\mathbb Z\setminus \{n\}}\{x+\i k\pi\colon x\leq\log  |c_n|\mp C_1\mp \log|k-n|\}\cup\{x+\i n\pi\colon x\leq\log  |c_n|\},
\end{eqnarray*}
and
 $\Gamma_-^{(n)}$  the family of curves $\gamma$ connecting $E_n$ with $E_n^-$
in $\Omega_n^-$ 
such that if $\gamma: (0,1)\to\mathbb C$ then we have $n\pi< \Im \gamma(t)<(n+1)\pi$ for $t$ sufficiently close to $0$  and $\Im \gamma(t)>(n+1)\pi$ for $t$ sufficiently close to $1$.  Similarly, let  $\Gamma_+^{(n)}$ be the family of curves $\gamma$ connecting $E_n$ with
$E_n^+$ in
$\Omega_n^+$ such that if $\gamma: (0,1)\to\mathbb C$ then we have $n\pi< \Im \gamma(t)<(n+1)\pi$ for $t$ sufficiently close to $0$  and $\Im \gamma(t)>(n+1)\pi$ for $t$ sufficiently close to $1$.
We have 
\begin{equation}
\label{schranken1}
\mod \Gamma_-^{(n)}\leq \mod \varphi(\Gamma_1^{(n)})\leq \mod \Gamma_+^{(n)}
 \end{equation}
 by the monotonicity properties of the module. 
Since all domains $\Omega_n^+$ ($\Omega_n^-$) and all sets $E_n^+$ ($E_n^-$) differ only by a translation, the upper (lower) bound in (\ref{schranken1}) is in fact independent of $n\in\mathbb Z$. By the conformal invariance of the module we have thus
\[0<C_{11}\leq \mod \Gamma_1^{(n)}\leq C_{12}<\infty,\]
and a similar estimate holds for $\Gamma_2^{(n)}$.
 From the subsequent Lemma \ref{modullemma2} we obtain the existence of $\delta_2\in (0,1)$ such that
\[\tilde{E}_n\subset\left[\frac{x_n+x_{n+1}}{2}-\delta_2\frac{x_{n+1}-x_n}{2}, \frac{x_n+x_{n+1}}{2}+\delta_2\frac{x_{n+1}-x_n}{2}\right], \qquad n\in\mathbb Z.\]
Since $\lambda_n\in\tilde{E}_n$, (\ref{sep1})  implies $\inf_{n\in\mathbb Z}|x_n-\lambda_n|>0$, $\inf_{n\in\mathbb Z}|x_{n+1}-\lambda_n|>0$, and $\delta:=\inf_{n\in\mathbb Z}|\lambda_n-\lambda_{n+1}|>0$. Hence the set $\{\lambda_n\}_{n\in\mathbb Z}$ is separated and condition (i) of Theorem \ref{pavlov} is fulfilled.
\end{proof}

 Condition (ii) of Theorem \ref{pavlov} is now quite easy to establish, since the limit
\[F(z)=F(0)\lim_{R\to\infty} \prod_{|\lambda_n|< R}\left(1-\frac{z}{\lambda_n}\right)\]
is known to exist for all functions $F$ of the \emph{Cartwright class}, i.e. entire functions $F$ of exponential type such that
\begin{equation}
\label{cw}
\int\limits_{-\infty}^{\infty}\frac{\log^+ |F(t)|}{1+t^2}dt<\infty,
\end{equation}
see \cite{lev96}. 

\begin{lem}
$F$ belongs to the Cartwright class. 
\end{lem}
\begin{proof}
We have  to show the existence of the integral (\ref{cw}).  Relation (\ref{sep}) implies
\[|\lambda_n|\geq -|\lambda_0|+|n|\delta, \qquad n\in\mathbb Z,\]
and (\ref{sep3}) implies that also (\ref{eq2}) holds. 
For $t\in(\lambda_{n-1}, \lambda_{n})$ we know that $\varphi (t)$ is  on  a slit of $\Omega(s)$ such that $\Im \varphi(t)=n\pi$ and $\Re\varphi (t)\leq\log|c_{n}|$.   Hence by (\ref{ugl})
\[\log |F(t)|=\Re\varphi (t)\leq  C_2+\log |n|,\]
and also 
\[\log^+ |F(t)|\leq C_2+\log |n|\]
for all $|n|\geq n_0$ such that the right hand side of this inequality is positive. We obtain
\begin{eqnarray*}
&&\int\limits_{\lambda_{n-1}}^{\lambda_{n}}
\frac{\log^+|F(t)|}{1+t^2}dt\leq (C_2+\log |n|)\int\limits_{\lambda_{n-1}}^{\lambda_{n}}
\frac{dt}{1+t^2}
\leq\frac{(C_2+\log |n|)(\lambda_{n}-\lambda_{n-1})}{1+\min(|\lambda_{n-1}|,|\lambda_{n}|)^2}\\
&&\leq\frac{(C_2+\log |n|)\Delta}{1+(-|\lambda_0|+ (|n|-1)\delta)^2}\leq C_{13}\frac{\log |n|}{|n|^2},
\end{eqnarray*}
provided that $|n|\geq n_1$ where $n_1\geq n_0$ is such that $-|\lambda_0|+(n_1-1)\delta>0$. Consequently, 
\[\int\limits_{-\infty}^{\infty}\frac{\log^+ |F(t)|}{1+t^2}dt\leq \int\limits_{\lambda_{-n_1}}^{\lambda_{n_1+1}}\frac{\log^+ |F(t)|}{1+t^2}dt
+C_{13}\sum_{|n|\geq n_1} \frac{\log |n|}{|n|^2}<\infty.
\]\end{proof}
Next we have to show condition  (iii) of Theorem \ref{pavlov}, i.e. the Muckenhoupt condition (\ref{A2}). For that purpose we choose $y<0$ such that the line $\i y+\mathbb R$ is disjoint with all semicircles  $H_{n,n+1} \,(n\in\mathbb Z)$, which is possible in view of (\ref{sep3}). Due to (\ref{sep}) there is an integer  $N>0$ such that the intervals $I_n+\i y$ where $I_n:=[(x_{n-1}+x_n)/2, (x_n+x_{n+1})/2]$  lie inside $H_{n-N, n+N}$ for all $n\in\mathbb Z$.  Using again the Gehring-Hayman Theorem  we find that there is $C_{14}>0$ with
\[|\Re(\varphi(x_n)-\varphi(z))|<C_{14},\qquad z\in I_n+\i y, n\in\mathbb Z.\]
This implies immediately 
\[|{\rm e}^{\varphi (z)} |\leq  |c_n| {\rm e}^{C_{14}}\qquad \mbox{and }\qquad |{\rm e}^{-\varphi (z)} |\leq  |c_n|^{-1}{\rm e}^{C_{14}}\]
for $z\in I_n+\i y, n\in\mathbb Z$.

Given any interval $I\subset \mathbb R$ we denote by $I_p, \ldots, I_q$ all intervals among $\{I_n\}_{n\in \mathbb Z}$ that are not disjoint with $I$.  Further we find
\begin{equation*}
\int_{I_n} |F(x+\i y)|^2\,dx \leq c_n^2{\rm e}^{2C_{14}}\Delta_1, \qquad \int_{I_n} |F(x+\i y)|^{-2}\, dx \leq c_n^{-2}{\rm e}^{2C_{14}}\Delta_1, \qquad n=p, \ldots, q.
\end{equation*}
Adding these inequalities for $n=p,\ldots, q$ and multiplying the results yields
\begin{equation*}
\begin{split} \int_{I} |F(x+\i y)|^2\,dx
\int_{I} |F(x+\i y)|^{-2}\,dx
& \leq \sum_{n=p}^q \int_{I_n} |F(x+\i y)|^2\,dx
\sum_{n=p}^q \int_{I_n} |F(x+\i y)|^{-2}\,dx
\\
&\leq {\rm e}^{4C_{14}}\Delta_1^2 \sum_{n=p}^q c_n^2\sum_{n=p}^q c_n^{-2}.
\end{split}
\end{equation*}
First assume $q-p>1$. Since at least $(q-p-1)$ of the intervals $I_p, \ldots , I_q$ are contained in $I$ we have $|I|\geq \delta_1 (q-p-1)\geq \frac{\delta_1}{3}{(q-p+1)}$, and using the assumption (\ref{A2t}) we get
\begin{equation*}
 \int_{I} |F(x+\i y)|^2\,dx
\int_{I} |F(x+\i y)|^{-2}\,dx
\leq {\rm e}^{4C_{14}}\Delta_1^2 C\frac{9}{\delta_1^2}|I|^2,
\end{equation*}
and thus (\ref{A2}).

In the special case $p=q$ the interval $I$ is contained in $I_p=I_q$ and we can estimate
\begin{equation*}
 \int_{I} |F(x+\i y)|^2\,dx
\int_{I} |F(x+\i y)|^{-2}\,dx
\leq c_p^2{\rm e}^{2C_{14}}|I| c_p^{-2}{\rm e}^{2C_{14}}|I|= {\rm e}^{4C_{14}}|I|^2.
\end{equation*}
If finally $q=p+1$ then the interval  $I$ is contained  in $I_p\cup I_{p+1}$. Using $c_p^2/(4C)\leq c_{p+1}^2\leq 4C c_p^2$ this case can be handled in the following way:
\begin{equation*}
\begin{split}
 \int_{I} |F(x+\i y)|^2\,dx
\int_{I} |F(x+\i y)|^{-2}\,dx
&\leq (c_p^2+c_{p+1}^2){\rm e}^{2C_{14}}|I| (c_p^{-2}+c_{p+1}^{-2}){\rm e}^{2C_{14}}|I| \\ &\leq (1+4C)^2 {\rm e}^{4C_{14}}|I|^2.
\end{split}
\end{equation*}
The Muckenhoupt condition is hence always fulfilled and Theorem \ref{pavlov} implies that $\{\lambda_n\}_{n\in\mathbb Z}$ is a complete interpolating sequence. 

\medskip

Next we have to show that every 
{\cis}  $\{\lambda_n\}_{n\in\mathbb Z}$ can be obtained in this way. 
From (ii) of Theorem \ref{pavlov} we know that the numbers  $\{\lambda_n\}_{n\in\mathbb Z}$ 
are the zeros of an entire function $F$ of exponential type $\pi$, which also belongs to the Laguerre-P\'{o}lya class ${\cal LP}$. We can assume that $(-1)^nc_n\geq 0$ holds for the sequence $\{c_n\}_{n\in\mathbb Z}={\rm Cr}\, F$, otherwise we consider $-F$.  Hence we have the representation (\ref{expf}) with a conformal map $\varphi:\mathbb C_-\to \Omega({\rm Cr}\, F)$. Letting $R\to \infty$  in 
\[\prod_{|\lambda_n|<R}\left |1-\frac{\i y}{\lambda_n}\right|^2=\prod_{|\lambda_n|<R}\left (1+\frac{ y^2}{\lambda_n^2}\right)\geq \sum_{|\lambda_n|<R}\frac{y^2}{\lambda_n^2}, \qquad y\in \mathbb R,\]
gives
\[|F(\i y)|\geq |y|\left(\sum_{n\in\mathbb Z}\frac{1}{\lambda_n^2}\right)^{1/2}, \qquad y\in \mathbb R,\]
and hence $\displaystyle \lim_{y\to -\infty}\Re \varphi(\i y)=\lim_{y\to-\infty}\log |F(\i y)|=\infty$. 

It remains to show condition (\ref{A2t}) for the sequence $\{d_n\}_{n\in\mathbb Z}=\{c_n^2\}_{n\in\mathbb Z}$. Fix $y<0$ and recall that $w(x)=|F(x+\i y)|^2$ satisfies the Muckenhoupt condition (\ref{A2}) with a constant $C>0$. Next we need that the distance between neighboring slit endpoints is bounded. 

\begin{lem}
\label{lem4}
There is $C_{15}>0$ such that
\begin{equation}
\label{nachbarn}
|\log |c_n|-\log |c_{n+1}||\leq C_{15},\qquad\forall n\in\mathbb Z.
\end{equation}
\end{lem}
\begin{proof}
 By Lemma 6.5 of chapter VI of \cite{gar},  the function $\psi(x):=\frac{1}{2} \log w(x)=\log |F(x+\i y)|=\Re \varphi (x+\i y)$ belongs to the space BMO$(\mathbb R)$, i.e.
\[\|\psi\|_*:=\sup_I\frac{1}{|I|}\int_I |\psi(x)-\psi_I|dx<\infty, \qquad \psi_I:= \frac{1}{|I|}\int_I\psi (x)dx,\]
where the supremum is taken over all bounded intervals.  This implies (Ibid.,Theorem 1.2) that
\[\int_{\mathbb R}\frac{\log^+|F(x+\i y)|}{1+x^2}dx\leq \int_{\mathbb R}\frac{|\psi(x)|}{1+x^2}dx <\infty, \]
and the representation 
\begin{equation}
\label{darst}
\displaystyle
\log |F (x)|=B|y|+\sum_{n\in\mathbb Z}\log\left |\frac{\displaystyle1-(x-\i y)/(\lambda_n-\i y)}{\displaystyle 1-(x-\i y)/(\lambda_n+\i y)}\right|+ \frac{1}{\pi}\int_{\mathbb R} \frac{|y|}{|x-\i y -t|^2}\log|F(t+\i y)|dt, \quad x\in \mathbb R,
\end{equation}
with $ B:=\limsup_{t\to \infty}\log |F(\i t)|/t$ is valid, cf. section III.G.3 of \cite{koo}. 
This formula shows among other things that $\Re \varphi$ restricted to $\mathbb R$ has logarithmic singularities in the points $\lambda_n $, whereas on $\mathbb R+\i y$ this function is continuous. Therefore we cannot compare $\Re \varphi$ on $\mathbb R$ and $\mathbb R+\i y$ directly, but have to resort to mean values.  Let $\delta>0$ be  the separation constant defined in (\ref{sep}) and  let $I_x$ ($x\in\mathbb R$) be the interval of length $\delta$ centered at $x$.  By Lemma 1.1(b) of chapter VI of \cite{gar}\footnote{Inequality (\ref{ersteugl}) can be interpreted as a discrete analogue of this lemma. That it holds is particularly delicate in view of  an example of Avdonin \cite{avd} who constructed  a {\cis} so that $\psi$  oscillates so much that $\sup_{|x_1-x_2|\leq 1}|\psi (x_1)-\psi (x_2)|=\infty$.} we know that
\begin{equation}
\label{diff}
|\psi_{I_{x_1}}-\psi_{I_{x_2}}|\leq C_{16}\log(1+\frac {1}{\delta}|x_1-x_2|)
\end{equation} 
for all $x_1, x_2\in\mathbb R$ and some constant $C_{16}>0$. Let $x\in\mathbb R\setminus\{\lambda_n\}_{n\in\mathbb Z}$ and assume $\lambda_{n-1}<x<\lambda_n$ for some $n\in\mathbb Z$. Then  $|x-\lambda_k|\geq \delta (k-n) $ for $k\geq n$ and $|x-\lambda_k|\geq \delta(n-k-1)$ for $k\leq n-1$, so that  we find
\begin{multline*}
0>\sum_{k\in\mathbb Z\setminus \{n-1,n\}}\log \left |\frac{\displaystyle1-(x-\i y)/(\lambda_k-\i y)}{\displaystyle1-(x-\i y)/(\lambda_k+\i y)}\right|=\sum_{k\in\mathbb Z\setminus \{n-1,n\}}\frac{1}{2}\log\left(1-\frac{4y^2}{(\lambda_k-x)^2+4y^2}\right) \\
\geq \sum_{k=1}^\infty\log \left(1-\frac{4y^2}{k^2\delta^2+4y^2}\right)=:C_{17}>-\infty.
\end{multline*}
If $I\subset [\lambda_{n-1},\lambda_n]$ is an interval we obtain
\begin{multline*}
0>\int_I \log\left |\frac{\displaystyle1-(x-\i y)/(\lambda_{n-1}-\i y )}{\displaystyle1-(x-\i y)/(\lambda_{n-1}+\i  y)}\right|+\log\left |\frac{\displaystyle1-(x-\i y)/(\lambda_n-\i y)}{\displaystyle1-(x-\i y)/(\lambda_n+\i y)}\right|\,dx\\ 
\geq \int_{\lambda_{n-1}}^{\lambda_n} \log\left(1-\frac{4y^2}{(\lambda_{n-1}-x)^2+4y^2}\right)+\log \left(1-\frac{4y^2}{(\lambda_n-x)^2+4y^2}\right)dx\\ 
=4 y \arctan \frac{\lambda_{n-1}-\lambda_n}{2y}-2(\lambda_{n-1}-\lambda_n)\log \left(1-\frac{4y^2}{(\lambda_{n-1}-\lambda_n)^2+4y^2}\right)\geq C_{18}>-\infty, \end{multline*}
where $C_{17}$ is independent of $I$ and $n$ since the function
\[\lambda\mapsto 4 y \arctan \frac{\lambda}{2y}-2\lambda \log \left(1-\frac{4y^2}{\lambda^2+4y^2}\right)\]
is bounded on $\mathbb R$. Altogether we have shown that for intervals $I$ not containing any $\lambda_n$ in its interior
\begin{equation}
\label{blaschke}
0>\frac{1}{|I|}\int_I \sum_{n\in\mathbb Z}\log\left |\frac{\displaystyle1-(x-\i y)/(\lambda_n-\i y)}{\displaystyle 1-(x-\i y)/(\lambda_n+\i y)}\right| dx\geq C_{17}+ \frac{1}{|I|} C_{18} .
\end{equation}
For any  interval $I_{\tilde{x}}$ we have
\begin{equation}
\label{means}
\begin{split}
&\frac{1}{ |I_{\tilde{x}}|}  \int_{I_{\tilde{x}}} \frac{1}{\pi} \int_{\mathbb R} \frac{|y|}{|x-\i y-t|^2} \log |F(t+\i y)|\,dt\, dx\\
&=\frac{1}{ \delta}  \int_{I_{\tilde{x}}} \frac{1}{\pi} \int_{\mathbb R} \frac{|y|}{|\tilde{x}-\i y-s|^2} \log |F(s+x-\tilde{x}+\i y)|\,ds \,dx\\
&=\frac{1}{ \pi}  \int_{\mathbb R}  \frac{|y|}{|\tilde{x}-\i y-s|^2}\frac{1}{\delta}\int_{I_{\tilde{x}}}  \log |F(s+x-\tilde{x}+\i y)|\,dx \,ds =\frac{1}{\pi}\int_{\mathbb R}  \frac{|y|}{|\tilde{x}-\i y-s|^2}\psi_{I_s}\,ds,
\end{split}
\end{equation}
and the latter integral can be estimated using (\ref{diff}) in the following way
\begin{eqnarray*}
&&\left | \frac{1}{\pi}\int_{\mathbb R}  \frac{|y|}{|\tilde{x}-\i y-s|^2}\psi_{I_s}\,ds-\psi_{I_{\tilde{x}}}\right|\\
&&=\left|\frac{1}{\pi}\int_{\mathbb R} \frac{|y|}{|\tilde{x}-\i y-s|^2}\psi_{I_s}\,ds-
\frac{1}{\pi}\int_{\mathbb R}  \frac{|y|}{|\tilde{x}-\i y-s|^2}\psi_{I_{\tilde{x}}}\,ds \right|\\
&&\leq \frac{1}{\pi}\int_{\mathbb R}\frac{|y|C_{16}\log (1+\frac{1}{\delta}|\tilde{x}-s|)}{|\tilde{x}-\i y -s|^2}\,ds
=\frac{C_{16}}{\pi}\int_{\mathbb R} \frac{|y|\log(1+t/\delta)}{t^2+y^2}dt=:C_{19}<\infty.
\end{eqnarray*}
Combining this with (\ref{darst}), (\ref{blaschke}), and (\ref{means}) we arrive at 
\[ \left|\frac{1}{|I_{\tilde{x}}|}\int_{I_{\tilde{x}}} \log |F(x)|\,dx-\psi_{I_{\tilde{x}}}\right|\leq B|y|+|C_{17}|+\frac{1}{\delta} | C_{18}|+C_{19}=:C_{20}\]
for all intervals $I_{\tilde{x}}$ of length $\delta$ centered at $\tilde{x}$ and not containing any $\lambda_n$ in its interior. Assume that $\tilde{x}_n$ is such that $I_{\tilde{x}_n}\subset [\lambda_{n-1}, \lambda_n]$.  Then we have $\Re \varphi(x)\leq \log |c_{n}|$ for all $x\in I_{\tilde{x}_n}$ and thus
\[\frac{1}{|I_{\tilde{x}_n}|}\int_{I_{\tilde{x}_n}}\log |F(x)|dx\leq \log |c_n|, \qquad n\in \mathbb Z,\]
and also
\begin{equation}
\label{lower}
\psi_{I_{\tilde{x}_n}}-C_{20}\leq\log |c_n|, \qquad n\in \mathbb Z.
\end{equation}
In order to get an upper bound  for $\log |c_n|$ we  recall that $\Re \varphi (x_n)=\log |c_n|$ for  some $x_n\in (\lambda_{n-1},\lambda_n)$ and hence (\ref{darst}) yields
\begin{equation}
\label{upper}
\log |{c_n}|\leq B |y|+\frac{1}{\pi}\int_{\mathbb R} \frac{|y|}{|x_n-\i y -t|^2}\log |F(t+\i y)|\,dt.
\end{equation}
Now we compute
\begin{multline*}
\left |\frac{1}{\pi}\int_{\mathbb R}\frac{|y|}{|x_n-\i y -t|^2}\log |F(t+\i y)|\,dt-\psi_{I_{\tilde{x}_n}}\right|
=\left |\frac{1}{\pi}\int_{\mathbb R}\frac{|y|}{|x_n-\i y -t|^2}\left(\psi(t)-\psi_{I_{\tilde{x}_n}} \right)\,dt \right|\\
 \leq \frac{1}{\pi}\sum_{k\in\mathbb Z}  \int\limits_{I_{\tilde{x}_n}+k\delta} \frac{|y|}{(x_n-t)^2+y^2} \left(\left |\psi(t)-\psi_{I_{\tilde{x}_n}+k\delta}\right| +\left| \psi_{I_{\tilde{x}_n}+k\delta}- \psi_{I_{\tilde{x}_n}} \right|\right)\,dt. 
\end{multline*}
For $t\in I_{\tilde{x}_n}+k\delta$ and $ k\geq 0$ we have  $t-x_n\geq (\lambda_{n-1}+k\delta)-\lambda_n\geq k\delta-\Delta$, where $\Delta$ is defined in  (\ref{eq2}). For $k\leq 0$ we find similarly $x_{n}-t\geq-k\delta-\Delta$. Thus for all $k$ with $|k| >k_0:=\Delta/\delta$ we have  $|x_n-t|\geq |k|\delta-\Delta>0$ and we can continue the previous chain of inequalities using (\ref{diff})
\begin{eqnarray*}
\ldots &\leq & \frac{1}{\pi}\sum_{|k|\leq k_0}\frac{\delta}{|y|}\left(\frac{1}{\delta}\int\limits_{I_{\tilde{x}_n}+k\delta}
\left |\psi(t)-\psi_{I_{\tilde{x}_n}+k\delta}\right|\,dt+ C_{16}\log(1+|k|)\right)
\\  &&+ \frac{1}{\pi}\sum_{|k| > k_0}\frac{|y|\delta}{(|k|\delta-\Delta)^2+y^2} 
\left(\frac{1}{\delta}\int\limits_{I_{\tilde{x}_n}+k\delta}
\left |\psi(t)-\psi_{I_{\tilde{x}_n}+k\delta}\right|\,dt+ C_{16}\log(1+|k|)\right)\\
&\leq & \frac{\|\psi\|_*}{\pi}\left(\sum_{|k|\leq k_0}\frac{\delta}{|y|}+\sum_{|k|>k_0}\frac{|y|\delta}{(|k|\delta-\Delta)^2+y^2}\right)
\\ &&+\frac{C_{16}}{\pi}\left(\sum_{|k|\leq k_0}\frac{\delta}{|y|}\log (1+|k|)+\sum_{|k|>k_0}\frac{|y|\delta\log (1+|k|)}{(|k|\delta -\Delta)^2+y^2}\right)=:C_{21}<\infty.
\end{eqnarray*} 
In view of  (\ref{upper}) we get thus the upper estimate
\[\log |c_n|\leq B|y|+\psi_{I_{\tilde{x}_n}}+ C_{21}, \qquad n\in\mathbb Z.\]
Considering also (\ref{lower}) we find
\begin{equation}
\label{ungl3}
 |\log |c_n|-\psi_{I_{\tilde{x}_n}}|\leq C_{22}, \qquad n\in\mathbb Z,
 \end{equation}
where $C_{22}:=\max (C_{20},B|y|+C_{21})$. Using once more (\ref{diff}) we conclude
(\ref{nachbarn})
with the constant $C_{15}:=2C_{22}+C_{16}\log(1+2\Delta/\delta)$, and the lemma is proved. 
\end{proof}

Condition  (\ref{eq2}) implies that $\sup_{n\in\mathbb Z}|x_{n+1}-x_n|\leq 2\Delta$, and hence $\mathbb R+\i y$ is disjoint with all semicircles $H_{n,n+1}$ connecting $x_n$ and $x_{n+1}$ in $\mathbb C_-$, provided that $y$ had been chosen so that $|y|>2\Delta$. Condition (\ref{sep}) makes it possible to choose $N>0$ so large that the intervals $[\lambda_{n-1}, \lambda_n]+\i y$ lie inside the semicircles $H_{n-N,n+N}$ for all $n\in\mathbb Z$. The Gehring-Hayman Theorem and (\ref{nachbarn}) give the estimate
\[{\rm length}\,(\varphi(H_{n,n+1}))\leq C_9\sqrt{\pi^2+C_{15}^2},\qquad
 {\rm length}\,(\varphi(H_{n-N,n+N}))\leq 2N C_9\sqrt{\pi^2+C_{15}^2}.
 \]
 Hence 
 \[|\Re (\varphi(x_n)-\varphi(z))|<C_{23}, \qquad z\in [\lambda_{n-1}, \lambda_n]+\i y, n\in\mathbb Z,\]
 with a constant $C_{23}>0$,
 and also
 \[|c_n|\leq {\rm e}^{C_{23}}\left | {\rm e}^{\varphi (z)}\right|, \quad 
 |c_n|^{-1} \leq {\rm e}^{C_{23}}\left | {\rm e}^{-\varphi (z)}\right|, \qquad 
 z\in [\lambda_{n-1}, \lambda_n]+\i y, n\in\mathbb Z.\]
 For  any set $I=\{p,\ldots, q\}$  of consecutive integers we exploit (\ref{A2}) for the interval $[\lambda_{p-1}, \lambda_q]$ to obtain
 \begin{eqnarray*}
 &&
 \sum_{n=p}^q c_n ^2\sum_{n=p}^q c_n^{-2}\leq
 \sum_{n=p}^q\frac{1}{\lambda_n-\lambda_{n-1}}\int\limits_{\lambda_{n-1}}^{\lambda_n}{\rm e}^{2C_{23}} \left |{\rm e}^{\varphi (x+\i y)}\right|^2\,dx \cdots\\
 &&
 \sum_{n=p}^q\frac{1}{\lambda_n-\lambda_{n-1}}\int\limits_{\lambda_{n-1}}^{\lambda_n}{\rm e}^{2C_{23}} \left |{\rm e}^{-\varphi (x+\i y)}\right|^2\,dx \leq
 \frac{{\rm e}^{4C_{23}}}{\delta^2} \int\limits_{\lambda_{p-1}}^{\lambda_q} w(x)dx \int\limits_{\lambda_{p-1}}^{\lambda_q} \frac{1}{w(x)} dx  \\
 &&\leq \frac{{\rm e}^{4C_{23}}}{\delta^2} C (\lambda_q-\lambda_{p-1})^2\leq \frac{{\rm e}^{4C_{23}}}{\delta^2}C \Delta^2 |I|^2,
  \end{eqnarray*}
  i.e. (\ref{A2t}) is established and we are done.
 \end{proof}

\section{Sufficient conditions for com\-plete inter\-pola\-ting se\-quen\-ces}
 As before  let  $\{c_n\}_{n\in\mathbb Z}$ be the sequence of critical values of an entire function in the Laguerre-P\'{o}lya class $ {\cal LP}$. It is easy to see that the condition  (\ref{erecond})   with  constants $C,c>0$ implies the discrete Muckenhoupt condition  (\ref{A2t}) for the numbers $d_n=c_n^2$. In view of the criterion of Eremenko and Sodin in Theorem \ref{eresodthm}, Theorem \ref{levinthm} is contained on Theorem \ref{charthm}. This can be regarded as the discrete analogue of the observation that Theorem \ref{pavlov} contains Theorem \ref{levinthm} as a special case, since  sine-type functions are easily seen to fulfill the Muckenhoupt condition (\ref{A2}). 
 
 On the other hand, the sequence 
 \[d_n=(1+|n|)^{2\alpha}, \qquad \mbox{$ -\frac{1}{2}<\alpha<\frac{1}{2},$}\]
 is not bounded away from $0$ (for $-\frac{1}{2}<\alpha<0$) or $\infty$ (for $0<\alpha<\frac{1}{2}$), but satisfies (\ref{A2t}).  Hence the generating function of a complete interpolating sequence is not always of sine-type. That the discrete Muckenhoupt condition is indeed fulfilled for this example follows from the subsequent sufficient condition. We  connect two functions by the symbol $\asymp$ if their quotient is always between certain positive constants. 
 
 \begin{lem}
 \label{wachstum}
 If the condition  $d_n\asymp(1+|n|)^{2 \alpha}$ holds for a sequence $\{d_n\}_{n\in\mathbb Z} $ and some   $\alpha\in\left (-\frac{1}{2},\frac{1}{2}\right)$, then the discrete Muckenhoupt condition {\rm  (\ref{A2t})} is fulfilled. 
 \end{lem}

\begin{proof}
The assumption   $d_n\asymp(1+|n|)^{2\alpha}$ means that there are constants $C_1, C_2>0$ such that
\[C_1(1+|n|)^{2\alpha}\leq d_n \leq C_2 (1+|n|)^{2\alpha}, \qquad \forall n\in \mathbb Z.\]
For every set $I=\{p,\ldots, q\}$ of consecutive integers we can estimate
\begin{eqnarray*}
\sum_{n\in I} d_n\sum_{n\in I}d_n^{-1}&\leq & C_2 \sum_{n =p}^q(1+|n|)^{2\alpha}\cdot C_1^{-1}\sum_{n=p}^q(1+|n|)^{-2\alpha}\\ &\leq &  C_2 C_1^{-1}\int\limits_{p-1}^{q+1} (1+|x|)^{2\alpha}dx \int\limits_{p-1}^{q+1}(1+|x|)^{-2\alpha}dx.
\end{eqnarray*}
First consider the case $0<p<q$. Using the left inequality proved in Lemma \ref{ungl} below we find 
\begin{eqnarray*}
\sum_{n\in I} d_n \sum_{n\in I}d_n^{-1}&\leq &\frac{C_2}{C_1 (1-2\alpha)(1+2\alpha)}[(2+q)^{1-2\alpha}-p^{1-2\alpha}][(2+q)^{1+2\alpha}-p^{1+2\alpha}]\\
&=&C_3 [(2+q)^2-p^{1-2\alpha} (2+q)^{1+2\alpha}-(2+q)^{1-2\alpha}p^{1+2\alpha}+p^2]\\
&\leq & C_3 [(2+q)^2-2(2+q)p+p^2]=C_3(q-p+2)^2\\
&\leq& 4C_3(q-p+1)^2=4C_3|I|^2
\end{eqnarray*}
In case $p\leq 0< q$ we use the right inequality of Lemma \ref{ungl} to obtain
\begin{eqnarray*}
\sum_{n\in I} d_n\sum_{n\in I}d_n^{-1}&\leq &
C_3 [(2+q)^{1-2\alpha}+(|p|+1)^{1-2\alpha}][(2+q)^{1+2\alpha}+(|p|+1)^{1+2\alpha}]\\
&=&C_3 [(2+q)^2+(1-p)^{1-2\alpha} (2+q)^{1+2\alpha}+(2+q)^{1-2\alpha}(1-p)^{1+2\alpha}\\
&& +(1-p)^2] \leq  2 C_3 [(2+q)^2+(1-p)^2]\leq 2 C_3(q-p+3)^2\\
&\leq&  18 C_3(q-p+1)^2=18 C_3|I|^2.
\end{eqnarray*}
The remaining cases are similar. 
  \end{proof}
  
  The preceding lemma leads to a sufficient condition for {\cis}s which is analogous to the following  generalization of Theorem \ref{levinthm}.
  \begin{thm}
A sequence $\{\lambda_n\}_{n\in\mathbb Z}\subset \mathbb R$ is a {\cis} if conditions {\rm (i), (ii)} of Theorem \ref{pavlov} hold as well as 
\[|F(x+\i y)|\asymp (1+|x|)^\alpha, \qquad -\frac{1}{2}<\alpha<\frac{1}{2}.\]
\end{thm}
Pavlov \cite{pav} attributes this statement to a private  communication of Katsnelson, and it is also contained in  more general results by  Avdonin \cite{avd} and Sedletskii \cite{sed75}. \cite{lev96} contains an example which shows that the theorem becomes false for $|\alpha|\geq \frac{1}{2}$. Of course, the case $\alpha=0$ reduces to Theorem \ref{levinthm}. 

Theorem \ref{charthm} and 
Lemma \ref{wachstum} imply the following discrete analogue. 
\begin{cor}
Let $s=\{c_n\}_{n\in\mathbb Z}$ be a sequence with $(-1)^nc_n\geq 0$ and
 \[|c_n|\asymp (1+|n|)^\alpha, \qquad -\frac{1}{2}<\alpha<\frac{1}{2},\]
 and let $\varphi:\mathbb C_-\to \Omega(s)$ be a conformal map with $\lim_{y\to-\infty} \Re (\i y)=\infty$. If  $\varphi $ is appropriately normalized then   the entire function $F$ in {\rm (\ref{expf})} has exponential type $\pi$, and 
$\{\lambda_n\}_{n\in\mathbb Z}=F^{-1}(0)$ is a {\cis}. 
\end{cor}
It is not difficult to see that this statement  becomes false for other values of $\alpha$ since also Lemma \ref{wachstum} fails in this case. 

Next  we remark that the condition $\lim_{y\to-\infty} \Re (\i y)=\infty$ in Theorem \ref{charthm} determines $\varphi$  up to conformal self-maps of $\mathbb C_-$ that fix $\mathbb \infty$.  Hence $\varphi(az+b), a>0, b\in\mathbb R$ describes the set of all conformal mappings of $\mathbb C_-$ onto $ \Omega(s)$ with this condition  if $\varphi$ is one such mapping.  Though we know that $F(z)={\rm e}^{\varphi(az+b)}$ has exponential type, it is not clear for which $a>0$ the type is equal to $\pi$. Obviously, the type does not depend on $b$, whereas it is a linear function of $a$. Hence there is exactly one $a>0$ such that the type of $F$ is equal to $\pi$.  In order to indentify this value we introduce the \emph{upper} and \emph{lower density} of a sequence $\{\lambda_n\}_{n\in\mathbb Z}$ by  
\[D^+:=\lim_{r\to \infty} \max_{x\in\mathbb R}\frac{|\{\lambda_n\}\cap [x-r,x+r)|}{2r}, \qquad 
D^-:=\lim_{r\to \infty} \min_{x\in\mathbb R}\frac{|\{\lambda_n\}\cap [x-r,x+r)|}{2r}.\]
 It follows from results of Landau \cite{lan} that $D^+=D^-=1$ is necessary (but not sufficient) for {\cis}s.  The densities of $F^{-1}(0)$, regarded as a function of $a$, are of the form const$/a$.  Hence the density conditions are sufficient to identify the right conformal mapping.
 \begin{cor}
 A sequence $\{\lambda_n\}_{n\in\mathbb Z}\subset \mathbb R$ is a {\cis} if and only if it is the zero set of $F$ in {\rm (\ref{expf})}, and $D^+=1$ ($D^-=1$), where $\varphi:\mathbb C_-\to\Omega(s)$ is a conformal map and $\{d_n\}_{n\in\mathbb Z}=\{c_n^2\}_{n\in\mathbb Z}$ satisfies {\rm (\ref{A2t})}.
  \end{cor} 
  
  Now it becomes clear what was meant by a parameterization of the set of {\cis}s by independent parameters. For any sequence $\{d_n\}_{n\in\mathbb Z}$ of positive real numbers satisfying (\ref{A2t}) there is a unique sequence $s=\{c_n\}_{n\in\mathbb Z}$ with $c_n^2=d_n$ and $(-1)^nc_n\geq 0$.  Let 
  $\varphi:\mathbb C_-\to\Omega(s)$ be a conformal map such that $D^+=D^-=1$ holds for $\{\lambda_n\}_{n\in\mathbb Z}:=\varphi^{-1}(-\infty)$. Then $\{\lambda_n\}_{n\in\mathbb Z}$ is uniquely determined up to shifts $\lambda\mapsto \lambda+b, b\in\mathbb R$. If we fix one value of the sequence $\{\lambda_n\}_{n\in\mathbb Z}$, this sequence is even unique.  Conversely, if we start with a {\cis} $\{\lambda_n\}_{n\in\mathbb Z}$, the generating function and  hence also the sequence of its   critical values is unique up to multiplication by a positive constant. Thus we have proved
  \begin{cor}
  Let $\lambda\in\mathbb R,\; d>0$. There is a one-to-one correspondence between all positive sequences $\{d_n\}_{n\in\mathbb R}$ with {\rm (\ref{A2t})} and $d_0=d$, and all {\cis}s $\{\lambda_n\}_{n\in\mathbb Z}$ with $\lambda_0=\lambda$. 
  \end{cor}

\section{Connection with the discrete Muckenhoupt condition  for values of the derivative}
 As mentioned in section \ref{sec1}, the discrete Muckenhoupt condition (\ref{A2t}) was studied by Lyubarskii and Seip \cite{lyuseip} for the numbers $|F'(\lambda_n)|^2$. It is intuitively clear, that there is a connection between the values $\{c_n\}_{n\in\mathbb Z}$ of $F$ at its critical points and the magnitude of the derivative at the zeros of $F$.  This will be made precise in the following lemma.
 
 \begin{lem}
 \label{equivlemma}
 Let $\{\lambda_n\}_{n\in\mathbb Z}\subset \mathbb R$ be a \cis, $F$ its generating function, and $\{c_n\}_{n\in\mathbb Z}$  its critical values. Then 
 \[|F'(\lambda_n )|\asymp |c_n|.\]
 \end{lem}
 \begin{proof}
 The proof  relies on  the same technique as Lemma \ref{lem4}. Subtract $\log |x-\lambda_n|$ on both sides of (\ref{darst}) and  then let  $x\to\lambda_n$. We obtain
 \begin{eqnarray*}
 \log |F'(\lambda_n)|&=&B|y|+\sum_{k\in\mathbb Z\setminus\{n\}}\frac{1}{2}\log \left(1-\frac{4y^2}{(\lambda_k-\lambda_n)^2+4y^2}\right)-\log(2|y|)\\&&+\frac{1}{\pi}\int_{\mathbb R} \frac{|y|}{|\lambda_n-\i y+t|^2}\log|F(t+\i y)|\,dt.
 \end{eqnarray*}
 Using $|\lambda_n-\lambda_k|\geq \delta|n-k|$, where $\delta$ ist the separation constant from (\ref{sep}), we obtain
 \begin{multline}
 \label{ungl1}
 \left| \log|F'(\lambda_n)|-\frac{1}{\pi}\int_{\mathbb R} \frac{|y|}{|\lambda_n-\i y+t|^2}\log|F(t+\i y)|\,dt\right|
 \\\leq B|y|+|\log (2|y|)|-\sum_{k=1}^\infty \log\left(1-\frac{4y^2}{k^2\delta^2+4y^2}\right) =:C_1.\end{multline}
 If $I_{\tilde{x}_n}\subset [\lambda_{n-1},\lambda_{n}]$ is again an interval of length $\delta$ centered at $\tilde{x}_n$, and $\psi(x)=\log |F(x+\i y)|$, a similar computation as in Lemma \ref{lem4} shows
 \begin{equation}
 \label{ungl2}
 \left| \frac{1}{\pi}\int_{\mathbb R} \frac{|y|}{|\lambda_n-\i y+t|^2}\log|F(t+\i y)|\,dt-\psi_{I_{\tilde{x}_n}}\right|  \leq C_2, \qquad n\in\mathbb Z, 
 \end{equation}
 for a positive constant $C_2$ independent of $n$. 
 From the estimate (\ref{ungl3}) in the proof of Lemma \ref{lem4}, (\ref{ungl1}), and (\ref{ungl2}) we get
 \[|\log |F'(\lambda_n)|-\log |c_n||\leq C_{22}+C_1+C_2,\]
hence the quotient  $ |F'(\lambda_n)|/ |c_n|$ lies between certain positive constants for all values of $n$.
   \end{proof}
   
   We remark that it is not necessary for the preceding lemma that the critical value $c_n$  is  assumed in the interval $(\lambda_{n-1}, \lambda_n)$. From (\ref{ersteugl})   follows that $|c_n|\asymp |c_{n+N}|$ for every fixed value of $N$, therefore the critical points  and zeros of $F$ only have to be both in ascending order.

Lemma \ref{equivlemma}  can be used to show a version of Pavlov's theorem involving the values $|F'(\lambda_n)|$. In contrast to condition (iii') in section \ref{sec1}, \emph{every} relatively dense subsequence $\{\lambda_{n_k}\}_{k\in\mathbb Z}$ can be considered, i.e. also the full sequence $\{\lambda_{n}\}_{n\in\mathbb Z}$\footnote{It seems that the arguments in \cite{lyuseip} show the same, only the formulation of the result is weaker.}.
   
   \begin{cor}
   A sequence  $\{\lambda_{n}\}_{n\in\mathbb Z}\subset \mathbb R$ is a {\cis} if and only if conditions {\rm (i), (ii)} of Theorem \ref{pavlov} hold  as well as   
   \begin{enumerate}
 \item[{\rm (iii'')}] For one  (and  then for every) relatively dense subsequence  $\{\lambda_{n_k}\}_{k\in\mathbb Z}$ and the values $d_k:=|F'(\lambda_{n_k})|^2$  the discrete Muckenhoupt condition {\rm (\ref{A2t})} holds. \end{enumerate}
     \end{cor}
     \begin{proof}
     
     We will not reprove the sufficiency of (i),(ii),(iii'') which is done in \cite{lyuseip}.  We only show the discrete Muckenhoupt condition for the sequence $\{d_k\}_{k\in\mathbb Z}$ if 
     $\{\lambda_{n_k}\}_{k\in\mathbb Z}$ is any relatively dense  subsequence .  We can assume that $n_k<n_{k+1}$ for all $k\in\mathbb Z$. The relative density yields
     \[\Delta_2:=\sup_{k\in \mathbb Z}|\lambda_{n_k}-\lambda_{n_{k+1}}|     <\infty,\]
 and thus 
 \[\delta |n_p-n_q|\leq |\lambda_{n_p}-\lambda_{n_q}|\leq |p-q|\Delta_2, \qquad \forall p,q\in \mathbb Z.\]
 For any finite set $I=\{p,\ldots, q\}$ of consecutive integers,  Theorem \ref{charthm} and Lemma \ref{equivlemma} imply
 \begin{multline*}
 \sum_{k\in I}|F'(\lambda_{n_k})|^2\sum_{k\in I} |F'(\lambda_{n_k})|^{-2}\leq C\sum_{k\in I}c_{n_k}^2\sum_{k\in I}c_{n_k}^{-2} \leq C\sum_{k=n_p}^{n_q} c_{k}^2\sum_{k=n_p}^{n_q}c_{k}^{-2}\\
 \leq C' (n_q-n_p+1)^2\leq C' ((q-p)\frac{\Delta_2}{\delta}+1)^2\leq C' \max \left (\frac{\Delta_2}{\delta},1\right)^2|I|^2.
  \end{multline*}
  \end{proof}

\section{Summary and open questions}
The main result of the this paper is Theorem \ref{charthm} which characterizes {\cis}s. In contrast to known characterizations of such sequences, the separation of the points and the exponential type  of the generating function follow automatically. The relation with conformal mappings makes the application  of distortion theorems possible, and the discrete version (\ref{A2t})  of the Muckenhoupt condition  (\ref{A2}) is easier to verify for concrete applications.  Unfortunately, our approach does not give an independent proof of these characterizations, but is based on a comparison argument with the classical Theorem \ref{pavlov}.  For this reason, one can expect to extend  this argument to known  generalizations of Pavlov's Theorem, for example the one given in \cite{lyuseip} for Paley-Wiener spaces $PW_\pi^p$.

 An entire function is said to be in  $PW_\pi^p,\,1<p<\infty $ if it belongs to $L^p$ on the real line, and a sequence $\{\lambda_n\}_{n\in\mathbb Z}$ is called  \emph{{\cis} for}  $PW_\pi^p$ if for every sequence $\{a_n\}\in l^p(\mathbb Z)$ the interpolation problem (\ref{interpol}) has a unique solution $f\in PW_\pi^p $. It is straightforward to see that a  characterization analogous to Theorem \ref{charthm} of these {\cis}s is valid,  if we impose the discrete Muckenhoupt condition 
\renewcommand{\theequation}{$\tilde{A}_p$}
 \begin{equation}
\sum_{n\in I} d_n \sum_{n\in I} d_n^{-1/(p-1)}\leq C|I|^p
 \end{equation}
on the sequence  $\{ d_n\}_{n\in\mathbb Z}=\{|c_n|^p\}_{n\in\mathbb Z}$. On the other hand, the restriction to real sequences can not so easily be disposed of. As long as the imaginary parts of $\{\lambda\}_{n\in\mathbb Z}$ are bounded we can refer to  Corollary 1 in section 8 of chapter 4 in \cite{young} which asserts that $\{{\rm e}^{\Re\lambda_n \i t}\}_{n\in\mathbb Z}$ is a Riesz basis in $L^2(-\pi, \pi)$ if and only if   $\{{\rm e}^{\lambda_n \i t}\}_{n\in\mathbb Z}$ is so. It would be very interesting to know if a representation by conformal mappings  of the generating function can also be found in the general case, where the reflection principle does not work  immediately. 
\renewcommand{\theequation}{\arabic{equation}} 

\section{Lemmas} 
In this section we give the proofs of some lemmas  that had been postponed in the main text.
\begin{lem}
\label{modullemma}
For the module $s(r_1, r_2)$  of the family of curves connecting the noncircular sides of $S(r_1, r_2)$ defined in {\rm  (\ref{s})} we have {\rm (\ref{b})}.
\end{lem}
\begin{proof}
 The principal value of the logarithm maps $S(r_1, r_2)$ conformally onto
 \[\tilde{S}(r_1, r_2):=\{x+\i y\colon  \log r_1 < x < \log r_2,\; \varphi_2(x)< y < \varphi_1(x)\},\]
where $\varphi_1(x):=\frac{\pi}{2}(1-C_4x{\rm e}^{-x})$ and $\varphi_2(x):=-\varphi_1(x)$, see Figure \ref{picture2}.

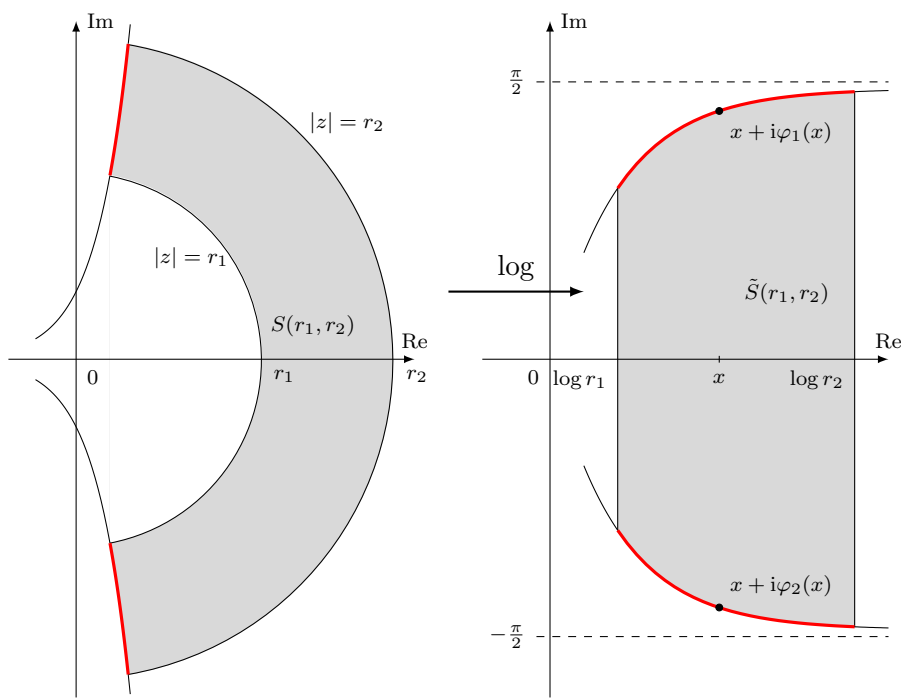
\begin{figure}[h]
\begin{center}

\begin{tikzpicture}[scale=0.9, >=latex]
%links
\filldraw[color=black!15!white] (0.77cm, -4.66cm ) arc (-80.627 : 80.627 :  4.73) (  0.77,   4.66) 
 --(0.6, 3.32)--  ( 0.5,  2.718)--  (0.5, -2.718)--(0.6 , -3.32)--( 0.77, -4.66 );
  \filldraw[white](0.5cm ,2.718cm ) arc (79.577:-79.577:2.763)--cycle;
  \draw (0.77cm,-4.66cm ) arc (-80.627:80.627:4.727 );
\draw (0.5cm, -2.718cm ) arc (-79.577:79.577:2.763 );

\draw[domain=-0.6:0.5] plot (\x,{ exp(2*\x)});
\draw[domain=0.5:0.77, very thick, red] plot (\x,{ exp(2*\x)});
\draw[domain=0.77:0.8] plot (\x,{ exp(2*\x)});
\draw[domain=-0.6:0.5] plot (\x,{ -exp(2*\x)});
\draw[domain=0.5:0.77, very thick, red] plot (\x,{ -exp(2*\x)});
\draw[domain=0.77:0.8] plot (\x,{ -exp(2*\x)});
%Achsen
\draw (-1,0) -- (0,0)node[anchor=north west]{\scriptsize $0$}--
(2.76,0)node[anchor=north west]{\scriptsize $r_1$};
\draw[->](2.76,0)--(4.72,0)node[anchor=north west]{\scriptsize $r_2$}--(5,0)node[anchor=south]{\scriptsize Re};
\draw[->] (0, -5) -- (-0, 5) node[anchor=west]{\scriptsize Im};
%Beschriftung
\node at (3.5,0.5){\scriptsize $S(r_1,r_2)$};
\node at (4,3.5){\scriptsize $|z|=r_2$};
\node at (1.7,1.5){\scriptsize $|z|=r_1$};
%Pfeil
\draw[->, thick] (5.5,1)--(7.5,1)node[midway, above]{\small $\log$};
\begin{scope}[xshift=7cm]

%rechts

\filldraw[domain=1:4.5, color=black!15!white] plot (\x,{4*(1-exp(-\x))})--(4.5, 0)--(1,0)--cycle;
\filldraw[domain=1:4.5, color=black!15!white] plot (\x,{-4*(1-exp(-\x))})--(4.5,0)--(1,0)--cycle;
\draw (1,2.528)--(1,0)node[anchor=north east]{\scriptsize $\log r_1$}--(1,-2.528);
\draw (4.5, 3.96)--(4.5,0)node[anchor=north east ]{\scriptsize $\log r_2$}--(4.5,-3.96);
\draw[dashed] (-0.2,4.1) node[anchor=east]{\scriptsize $\frac{\pi}{2}$} -- (5,4.1);
\draw[dashed] (-0.2,-4.1) node[anchor=east]{\scriptsize $-\frac{\pi}{2}$} -- (5,-4.1);

%Kurven
\draw[domain=0.5:1] plot (\x,{4*(1-exp(-\x))});
\draw[domain=1:4.5, very thick, red] plot (\x,{-4*(1-exp(-\x))});
\draw[domain=4.5:5] plot (\x,{4*(1-exp(-\x))});
\draw[domain=0.5:1] plot (\x,{-4*(1-exp(-\x))});
\draw[domain=1:4.5, very thick, red] plot (\x,{4*(1-exp(-\x))});
\draw[domain=4.5:5] plot (\x,{-4*(1-exp(-\x))});
%Punkte
\draw (2.5,-0.05)node[anchor=north]{\scriptsize $x$}--(2.5,0.05);
\draw[fill, color=black]   (2.5,3.67) circle (0.05cm) node[anchor=north west] {\scriptsize $x+\i\varphi_1 (x)$};
\draw[fill, color=black]   (2.5,-3.67) circle (0.05cm) node[anchor=south west] {\scriptsize $x+\i\varphi_2 (x)$};
\node at (3.5,1) {\scriptsize $\tilde{S}(r_1, r_2)$};
%Achsen
\draw[->]  (-1,0) -- (0,0)node[anchor=north east]{\scriptsize $0$}--(5,0) node[anchor=south]{\scriptsize Re};
\draw[->] (0, -5) -- (-0, 5) node[anchor=west]{\scriptsize Im};
\end{scope}

\end{tikzpicture}

\caption{Conformal mapping of $S(r_1, r_2)$ onto $\tilde{S}(r_1, r_2)$}
\label{picture2} 
\end{center}
\end{figure}

  By the conformal invariance of the module we know that
\begin{equation}
\label{eq6}
s(r_1, r_2)= \tilde{s}(r_1, r_2),
\end{equation}
where $\tilde{s}(r_1, r_2)$ is the module of the family of   all curves joining $\{x+\i\varphi_1(x)\colon \log r_1\leq x\leq \log r_2\}$ and $\{x+\i \varphi_2(x)\colon \log r_1\leq x \leq  \log r_2\}$ in $\tilde{S}(r_1, r_2)$. 
Denoting $\vartheta(x):=\varphi_1(x)-\varphi_2(x)=\pi(1-C_4x{\rm e}^{-x})$ and applying formula (13.4) from \cite{rowa} we get the estimate
\begin{equation}
\label{eq3}
\int\limits_{\log r_1}^{\log r_2} \frac{dx}{\vartheta (x)}\leq \tilde{s}(r_1, r_2)\leq \int\limits_{\log r_1}^{\log r_2} \frac{dx}{\vartheta (x)}+ R(r_1, r_2),
\end{equation} 
where
\begin{equation}
\label{r}
0\leq R(r_1, r_2)\leq \int\limits_{\log r_1}^{\log r_2} \frac{\varphi'_1(x)^2+\varphi'_2(x)^2}{\vartheta (x)}dx.
\end{equation}
We compute
\begin{equation}
\label{eq4}
\left |\, \int\limits_{\log r_1}^{\log r_2} \frac{dx}{\vartheta (x)}-\frac{1}{\pi}\log\frac{r_2}{r_1}\right|=\left|\frac{1}{\pi}
\int\limits_{\log r_1}^{\log r_2} \frac{dx}{1-C_4x{\rm e}^{-x}} - \frac{1}{\pi}
\int\limits_{\log r_1}^{\log r_2} dx\right|= \frac{1}{\pi}
\int\limits_{\log r_1}^{\log r_2} \frac{C_4x{\rm e}^{-x}}{1-C_4x{\rm e}^{-x}}dx\to 0
 \end{equation}
as $r_1, r_2\to \infty$. On the other hand, (\ref{r}) yields 
\begin{equation}
\label{eq5}
R(r_1, r_2)\to 0 \qquad \mbox{ as  $r_1, r_2\to \infty,$}
\end{equation}
since $\displaystyle  \int\limits_{x_0}^{\infty } \frac{\varphi'_1(x)^2+\varphi'_2(x)^2}{\vartheta (x)}dx$ is readily seen to converge.  Now (\ref{eq6}), (\ref{eq3}), (\ref{eq4}) and (\ref{eq5}) imply the assertion (\ref{b}). 
\end{proof}

Let $\mod(E_1, E_2)$ denote the module of the family of curves connecting two sets $E_1, E_2$ in $\mathbb C_-$. Then we have the following lemma.
\begin{lem}
\label{modullemma2}
For every $\varepsilon>0$ there is $\delta\in (0,1)$ with the property that for all $-1<a<b<1$ with $\varepsilon <\mod([a,b],[1,\infty))<1/\varepsilon$ and $\varepsilon <\mod([a,b],(-\infty, -1])<1/\varepsilon$ we have $|a|,|b|\leq\delta$.
\end{lem}
\begin{proof}
As in \cite{ahl} we denote by $\Lambda(R)$ the extremal distance of $[-1,0]$ from $[R,\infty )$ in $\mathbb C$. The module of all curves joining $[-1,0]$ and $[R,\infty)$ in $\mathbb C_-$ is then $1/(2\Lambda(R))$, and we get
\begin{equation*}
\mod([a,b],[1,\infty))=\frac{1}{\displaystyle 2\Lambda\left(\frac{1-b}{b-a}\right)},\qquad
\mod([a,b],(-\infty,-1])=\frac{1}{\displaystyle 2\Lambda\left(\frac{1+a}{b-a}\right)}.
\end{equation*}
Since  $\Lambda(R)$ is monotonical, tends to $+\infty$ as $R\to+\infty$, and to $0$ as $R\to 0$,  our assumptions imply that there are constants $C, c>0$ depending only on $\varepsilon$  such that 
\[c\leq \frac{1-b}{b-a}\leq C, \qquad c\leq \frac{1+a}{b-a}\leq C.\]
Adding both  inequalities yields after elementary manipulations
\[ b-a\geq \frac{2}{2C+1}.\]
Hence
\[1-b\geq \frac{2c}{2C+1}, \qquad a+1\geq\frac{2c}{2C+1},\]
and the assertion is true if we put $\delta:=1-(2c)/(2C+1)$. 
\end{proof}

 \begin{lem}
 \label{ungl}
 For real numbers $p,q,\alpha$ with  $p,q>0$ and $-\frac{1}{2}\leq\alpha\leq\frac{1}{2}$ holds
 \[2pq\leq p^{1+2\alpha}q^{1-2\alpha}+p^{1-2\alpha}{q^{1+2\alpha}}\leq p^2+q^2.\] 
 \end{lem}
 \begin{proof}
We fix $p$ and $q$ and consider the function $g(\alpha):=p^{1+2\alpha}q^{1-2\alpha}+p^{1-2\alpha}{q^{1+2\alpha}}$. We have $g(0)=2pq$,  $g(\frac{1}{2})=p^2+q^2$, and computation of $g'(\alpha)$ shows that $g(\alpha)$ grows monotonically on $[0,\frac{1}{2}]$. For negative $\alpha$ the result follows from $g(\alpha)=g(-\alpha)$. 
 
 An elementary proof can be given using the  arithmetic-geometric means inequality for the left estimate and the rearrangement inequality \cite{eng} for the right estimate.
    \end{proof}

\vspace{2ex}
\noindent {\bf Acknowledgement.} I thank  A.~E. Eremenko for a discussion of \cite{eresod} and  many useful suggestions.  MATLAB and the Schwarz-Christoffel toolbox by T. Driscoll \cite{dri} were helpful  in  the creation of the figures. 

\bibliographystyle{amsplain}
\bibliography{references.bib}

\end{document}